\documentclass[reqno,11pt]{article}
\usepackage{mathrsfs}
\usepackage{amssymb}
\usepackage{amsthm}
\usepackage{amsmath}
\usepackage{amsfonts}
\usepackage{mathtools}
\usepackage{enumerate}
\usepackage{bm}

\usepackage{graphicx}

\usepackage{bbm}
\usepackage{mathrsfs}
\usepackage{amssymb}
\usepackage[thicklines]{cancel}

\UseRawInputEncoding

\usepackage[usenames,dvipsnames]{xcolor}
\usepackage{soul}

\usepackage{txfonts}

\usepackage{psfrag}

\usepackage{setspace,color,wrapfig}
\usepackage[colorlinks,linkcolor=blue,citecolor=cyan]{hyperref}

\numberwithin{equation}{section}
\newtheorem{theorem}{Theorem}[section]
\newtheorem{lemma}[theorem]{Lemma}

\newtheorem{problem}[theorem]{Problem}
\theoremstyle{definition}

\theoremstyle{remark}

\begin{document}
\title{  Subsonic  flows with a contact discontinuity in a two-dimensional  finitely long curved nozzle }
\author{Shangkun Weng\thanks{School of Mathematics and Statistics, Wuhan University, Wuhan, Hubei Province, 430072, People's Republic of China. Email: skweng@whu.edu.cn}\and Zihao Zhang\thanks{School of Mathematics and Statistics, Wuhan University, Wuhan, Hubei Province, 430072, People's Republic of China. Email: zhangzihao@whu.edu.cn}}
\date{}
\maketitle
\newcommand{\de}{{\mathrm{d}}}
\def\div{{\rm div\,}}
\def\curl{{\rm curl\,}}
\newcommand{\ro}{{\rm rot}}
\newcommand{\sr}{{\rm supp}}
\newcommand{\sa}{{\rm sup}}
\newcommand{\va}{{\varphi}}
\newcommand{\me}{\mathcal{M}}
\newcommand{\ml}{\mathcal{V}}
\newcommand{\md}{\mathcal{D}}
\newcommand{\mg}{\mathcal{G}}
\newcommand{\mh}{\mathcal{H}}
\newcommand{\mf}{\mathcal{F}}
\newcommand{\ms}{\mathcal{S}}
\newcommand{\mt}{\mathcal{T}}
\newcommand{\mn}{\mathcal{N}}
\newcommand{\mb}{\mathcal{P}}
\newcommand{\mm}{\mathcal{B}}
\newcommand{\mj}{\mathcal{J}}
\newcommand{\mk}{\mathcal{K}}
\newcommand{\my}{\mathcal{U}}
\newcommand{\mw}{\mathcal{W}}
\newcommand{\mq}{\mathcal{Q}}
\newcommand{\ma}{\mathcal{L}}
\newcommand{\mc}{\mathcal{C}}
\newcommand{\mi}{\mathcal{I}}
\newcommand{\n}{\nabla}
\newcommand{\e}{\tilde}
\newcommand{\m}{\omega}
 \newcommand{\q}{{\rm R}}
\newcommand{\p}{{\partial}}
\newcommand{\z}{{\varepsilon}}
\renewcommand\figurename{\scriptsize Fig}
\begin{abstract}
 In this paper,  we establish the existence and uniqueness of subsonic flows with a contact discontinuity in a two-dimensional  finitely long slightly curved nozzle by prescribing the entropy, the Bernoulli's quantity and  the horizontal mass flux distribution at the entrance and the flow angle at the exit.  The problem is formulated as
  a nonlinear  boundary value problem for a hyperbolic-elliptic mixed system with a free boundary.   The Lagrangian transformation is employed  to straighten  the contact discontinuity and  the Euler system is reduced to a second-order nonlinear elliptic equation for the stream function.   One of the key points in the analysis is to   solve   the associated
   linearized  elliptic boundary value
problem with mixed boundary conditions in a  weighted H\"{o}lder space.  Another one is to  employ the implicit function theorem to locate  the contact discontinuity.
\end{abstract}
\begin{center}
\begin{minipage}{5.5in}
\end{minipage}
\end{center}
\section{Introduction  }\noindent
\par In this paper,   we concern subsonic  flows with a contact discontinuity  in a finitely long slightly
curved nozzle, which can be described by  the two-dimensional steady compressible   Euler system:
\begin{align}\label{1-1}
\begin{cases}
\p_{x_1}(\rho u_1)+\p_{x_2}(\rho u_2)=0,\\
\p_{x_1}(\rho u_1^2)+\p_{x_2}(\rho u_1 u_2)+\p_{x_1} P=0,\\
\p_{x_1}(\rho u_1u_2)+\p_{x_2}(\rho u_2^2)+\p_{x_2} P=0,\\
\p_{x_1}(\rho u_1 E+u_1 P)+\p_{x_2}(\rho u_2 E+u_2 P)=0,\\
\end{cases}
\end{align}
  where ${\bf u}=(u_1,u_2)$ is the velocity, $\rho$ is the density, $P$ is the pressure, $E$ is the
 energy, respectively.
 For polytropic gas, the equation of state and the
energy are of the form
\begin{equation*}
P= A(S)\rho^{\gamma}, \quad{\rm {and}}\quad E=\frac{1}{2}|{\bf u}^2|+\frac{ P}{(\gamma-1)\rho},
\end{equation*}
where  $A(S)= R e^{S}$ and $\gamma\in (1,+\infty)$, $R$ are positive constants. Denote the Bernoulli's function and the local sonic speed by $B=\frac12|{\textbf{u}}|^2+\frac{\gamma P}{(\gamma-1)\rho}$  and $ c(\rho,A)=\sqrt{A\gamma} \rho^{\frac{\gamma-1}{2}}$, respectively. It is well-known that  the system \eqref{1-1} is hyperbolic for supersonic flows ($ |\textbf{u}|>c(\rho,A) $), and hyperbolic-elliptic mixed for subsonic flows ($ |\textbf{u}|<c(\rho,A) $).
\par   We first introduce the notation of weak solution of the full Euler system \eqref{1-1}.  Suppose that a domain $\md $ in $ \mathbb{R}^2 $
is divided by a $C^1$-curve  $ \Gamma$
 into sub-domains $\md^+ $ and $\md^- $ such that $ \md= \md^- \cup\Gamma \cup \md^+ $.
Assume that $ \bm U = (\rho, u_1, u_2,P)$ is a $C^1$ solution of Euler equation \eqref{1-1} in each domain $ \md^+ $ and $ \md^- $ and is continuous up to the boundary $ \Gamma $ in each sub-domain. If   $ \bm U $ is a weak solution for \eqref{1-1} in domain $ \md $, by integration by parts, we see that  the  Rankine-Hugoniot conditions hold along $\Gamma$  almost everywhere:
\begin{equation}\label{1-2}
\begin{cases}
n_1[\rho u_1]+n_2[\rho u_2]=0,\\
n_1[\rho u_1^2]+n_1[P]+n_2[\rho u_1u_2]=0,\\
n_1[\rho u_1u_2]+n_2[\rho u_2^2]+n_2[P]=0,\\
n_1[ \rho u_1(E+\frac{P}{\rho})]+n_2[\rho u_2( E+\frac{P}{\rho})]=0,\\
\end{cases}
\end{equation}
where $ \mathbf{n} = (n_1, n_2) $  is the unit normal vector to $\Gamma$, and $[F](\mathbf{x}) = F_+(\mathbf{x})- F_-(\mathbf{x})$
denotes the jump across the  curve $\Gamma$ for a piecewise smooth function $ F $.
 Let $ \bm \tau= (\tau_{1},\tau_{2})$ as the unit tangential vector to $\Gamma$, which means that $ \mathbf{n}\cdot{\bm \tau}= 0$. Taking the dot product of $(\eqref{1-2}_2,\eqref{1-2}_3)$ with $ \mathbf{n}$  and $ \bm \tau $ respectively, we have
\begin{equation}\label{1-3}
\rho(\mathbf{u}\cdot\mathbf{n})[\mathbf{u}\cdot \bm \tau]_{\Gamma}= 0, \quad
[\rho(\mathbf{u}\cdot\mathbf{n})^2+P]_{\Gamma}= 0.
\end{equation}
 Assume that $ \rho> 0 $  in $ \bar \md $, \eqref{1-3} is divided into  two subcases:
 \begin{itemize}
 \item  $ \mathbf{u}\cdot\mathbf{n} \neq 0 $ and $[\mathbf{u}\cdot \bm \tau]_{\Gamma}=0 $. In this case,  the curve $ \Gamma $ is a shock;
  \item    $ \mathbf{u}\cdot\mathbf{n}=0 $ and $ [P]=0 $ on $\Gamma$. In this case, the curve $ \Gamma $ is  a contact discontinuity.
 \end{itemize}
\par Admissible shocks, centered rarefaction waves and
contact discontinuities are fundamental wave patterns in quasi-linear hyperbolic conservation laws.  The stability analysis of these elementary waves is very important both in physics and mathematics.  In the book \cite{CF48},  Courant and Friedrichs had presented some  nonlinear phenomenas for steady compressible flows, such as Mach reflection,  jet
flows, and their interactions. All these phenomenas are formulated by  elementary waves including contact discontinuities. So to understand the stability of contact discontinuity is an important step in studying these phenomenas.
 \par In recent years, there have been many
interesting results on the contact discontinuity   for various
physical situations.   The stability  of subsonic flows with a contact discontinuity  in infinite  nozzles was established in  \cite{BM09} and \cite{ BP19,PB19} with a Helmholtz decomposition.    The global existence and uniqueness  of the subsonic contact discontinuity with large vorticity in infinity long  nozzles were obtained  in \cite{CHWX19} by the theory of compensated compactness, which is not a perturbation around piecewise constant solutions.   The stability of supersonic flat contact discontinuity and transonic flat contact discontinuity for 2-D steady Euler flows in a finite nozzle was established in \cite{HFWX19,HFWX21}.  The stability  of two-dimensional  transonic contact discontinuity over a solid wedge  and three-dimensional supersonic and  transonic contact discontinuity were established in \cite{CYK13,CKY13} and \cite{WY15,WY13,WF15}.  The contact discontinuity in the Mach
reflection was also studied in \cite{CF07,CHF13}.  Recently,  the existence, uniqueness and stability of
subsonic flows past an airfoil with a vortex line were established in \cite{CXZ22}. The  well-posedness theory
for the steady subsonic jet flow
with a shock and a contact discontinuity was established in \cite{PW22}. 
\par   The contact discontinuity is part of the  solution  and is unknown, thus  it is a free boundary that separates the subsonic flow into two layers of the nozzle. Since the tangent of the  contact discontinuity is parallel to the  velocity of the flow  on its both sides, we can employ  the Lagrangian transformation to straighten  the free boundary.  Then by introducing a stream function and solving the
hyperbolic equations for the entropy and the Bernoulli's quantity, the Euler system  is reduced to a nonlinear second-order  equation for the stream function  in the subsonic region.  The
main ingredient   is to   solve   the associated
   linearized  elliptic boundary value
problem  with mixed boundary conditions
in a weighted H\"{o}lder space.
\par The other key ingredient in our  analysis is to use the implicit function theorem  to locate the contact discontinuity. This idea is motivated by the discussion of the airfoil problem in \cite{CXZ22}. For the problem of subsonic flows past an airfoil, the vortex line attached to the trailing edge is a contact discontinuity and is treated as a free boundary, the authors employed the implicit function theorem to solve this problem.   In this paper, we modified the approach in \cite{CXZ22} to study our problem. We choose a suitable H$\ddot{\rm{o}}$lder space and design a proper  map to verify the properties of the differential of the map. Then  the contact discontinuity can be located by the implicit function theorem.
\par This paper will be arranged as follows. In Section 2, we formulate the problem in detail and state the main result. In Section 3, we  reformulate the    problem in the Lagrangian coordinates  and   state   the main steps to solve the free boundary problem 3.1.
In Section 4, we  linearize the nonlinear  problem and solve the linear equation in a suitable weighted H$\ddot{\rm{o}}$lder space.
 In Section 5, we choose a suitable  H$\ddot{\rm{o}}$lder space and design a proper  map to verify the conditions in the implicit function theorem. In Section 6, we finish the proof of the main theorem.
   \section{Mathematical problem and the main result}\noindent
 \par In this section, we give a detailed formulation of subsonic flows with a contact discontinuity and state the main result.   Firstly, we consider a special class of  subsonic flows with a straight contact discontinuity in a finitely long flat nozzle.
\par The  flat nozzle (Fig 1) of the length $ L $ is given by
\begin{equation*}
\Omega_b:=\{(x_1,x_2):0<x_1<L,\ -1<x_2<1\}.
\end{equation*}
\begin{figure}
  \centering
  \includegraphics[width=11cm,height=5cm]{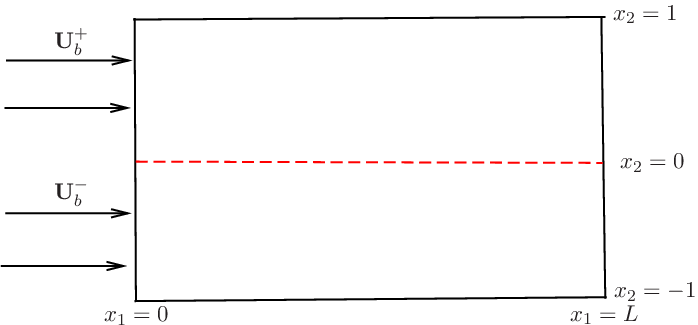}
  \caption{Subsonic flows with a contact discontinuity in a finitely long flat nozzle}
\end{figure}
Consider two layers of steady  Euler flows separated by the line $ x_2=0 $ satisfying the  properties:
     \begin{enumerate}[(i)]
\item The velocity and density of the top and bottom layers are given by $ (u_b^{+},0),\rho_b^{+} $ and $ (u_b^{-},0),\rho_b^{-} $, where $ u_b^\pm>0 $ and   $\rho_b^\pm>0$;
  \item the pressure of both the top and bottom layers  is given by the same positive constant $ P_b $;
\item the  flows in the   top and bottom layers are subsonic, i.e.,
 \begin{equation*}
 (u_b^\pm)^2<\frac{\gamma P_b}{\rho_b^\pm}.
\end{equation*}
\end{enumerate}
Then
\begin{equation}\label{2-1}
 \bm{U}_b=
  \begin{cases}
   \bm{U}_b^{+}:=(\rho_b^{+},u_{b}^{+},0,P_b),  \quad {\rm{for}}\quad (x_1,x_2)\in (0,L)\times (0,1),\\
   \bm{U}_b^{-}:=(\rho_b^{-},u_{b}^{-},0,P_b),  \quad {\rm{for}}\quad (x_1,x_2)\in (0,L)\times (-1,0),\\
  \end{cases}
 \end{equation}
with a contact discontinuity on the line $ x_2=0 $  satisfy  the steady Euler system \eqref{1-1} in the sense of weak solution,
  which will be called the background solutions in this paper.
  This paper is going to establish the structural
stability of these background  solutions  under the perturbations of suitable boundary conditions on the entrance and exit  and the upper and lower nozzle walls.
\par  The two-dimensional finitely long curved  nozzle $ \Omega $ (Fig 2) is  described by
\begin{equation}\label{2-2}
\Omega:=\{(x_1,x_2):0<x_1<L, g^-(x_1)<x_2<g^+(x_1)\},
\end{equation}
 where  $g^\pm(x_1)$ are small perturbations of the straight walls $ x_2=\pm 1 $, respectively. Furthermore,   $g^\pm(x_1)$ satisfy
  \begin{equation*}
  g^\pm(x_1)\in C^{2,\alpha}([0,L]) \quad {\rm{and}}\quad g^\pm(0)=\pm1.
  \end{equation*}
  The  upper and lower  boundaries of the nozzle are denoted by $ \Gamma_w^+ $ and $ \Gamma_w^- $, i.e;
\begin{equation}\label{2-3}
\Gamma_w^\pm:=\{(x_1,x_2):  0<x_1<L,\ x_2=g^\pm(x_1)\}.
\end{equation}
\begin{figure}
  \centering
  \includegraphics[width=12cm,height=5cm]{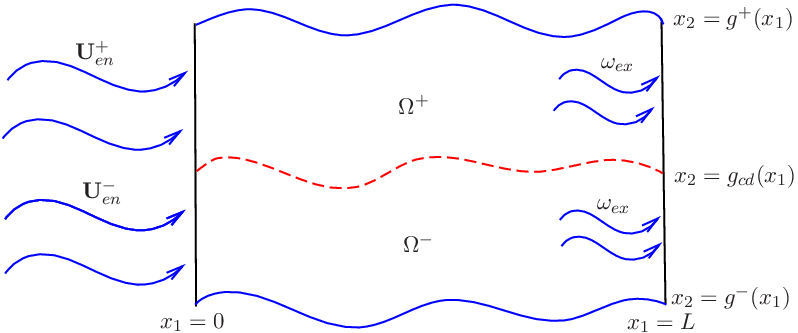}
  \caption{Subsonic flows with a contact discontinuity in a finitely long  curved nozzle}
\end{figure}
 The exit of the nozzle is denoted by
\begin{equation}\label{2-4}
\Gamma_{L}:=\{(x_1,x_2):   x_1=L,\ g^-(L)<x_2<g^+(L)\}.
\end{equation}
The entrance of the nozzle  is separated into two parts:
\begin{equation}\label{2-5}
\begin{aligned}
\Gamma_{0}^+:=\{(x_1,x_2):   x_1=0,\ 0<x_2<1\},   \quad \Gamma_{0}^-:=\{(x_1,x_2):  x_1=0,\ -1<x_2<0\}.
\end{aligned}
\end{equation}
\par At the entrance $ \Gamma_0^\pm $, we prescribe the boundary data   for  the entropy  $A $,  the Bernoulli's quantity $ B $ and  the horizontal mass distribution  $ J=\rho u_1 $:
\begin{equation}\label{2-6}
(A,B,J)(0,x_2)=
\begin{cases}
 \bm {U}_{en}^+(x_2):=(A_{en}^+,B_{en}^+,J_{en}^+)(x_2),\quad {\rm{on}}\quad \Gamma_0^+,\\
  \bm {U}_{en}^-(x_2):=(A_{en}^-,B_{en}^-,J_{en}^-)(x_2),\quad {\rm{on}}\quad \Gamma_0^-.
  \end{cases}
  \end{equation}
  At the exit, the flow angle $ \omega=\frac{u_2}{u_1}$ is prescribed by
 \begin{equation}\label{2-7}
  \omega(L,x_2)= \omega_{ex}(x_2), \quad {\rm{on}}\quad \Gamma_{L}.
  \end{equation}
  Here $ (\bm {U}_{en}^+,\bm {U}_{en}^-)(x_2)\in \left(C^{1,\alpha}([0,1])\right)^3\times \left(C^{1,\alpha}([-1,0])\right)^3 $ and $\omega_{ex}(x_2)\in C^{2,\alpha}([g^-(L),g^+(L)]) $ are close to the background solution in some sense that will be described later.
\par  We expect the flow in the nozzle  will be separated by a  contact discontinuity   $ \Gamma:=\{x_2=g_{cd}(x_1),0<x_1<L\} $ with $ g_{cd}(0)=0 $, and  denote
\begin{equation*}
\Omega^+:=\Omega\cap\{g_{cd}(x_1)<x_2<g^+(x_1)\}, \quad \Omega^-:=\Omega\cap\{ g^-(x_1)<x_2<g_{cd}(x_1)\}.
\end{equation*}
 Let
\begin{equation}\label{2-8}
{\bm {U}}(x_1,x_2)=
  \begin{cases}
   {\bm {U}}^+(x_1,x_2):=(\rho^+,u_1^+,u_{2}^+,P^+)(x_1,x_2)\quad {\rm{in}}\quad   \Omega^+,\\
   {\bm {U}}^-(x_1,x_2):=(\rho^-,u_1^-,u_{2}^-,P^-)(x_1,x_2)  \quad {\rm{in}}\quad  \Omega^-.\\
  \end{cases}
 \end{equation}
  Along the contact discontinuity $ x_2=g_{cd}(x_1) $, the following Rankine-Hugoniot conditions hold:
\begin{equation}\label{2-9}
\frac{u_2^+}{u_1^+}=\frac{u_2^-}{u_1^-}=g_{cd}^\prime(x_1), \quad P^+=P^-, \quad {\rm{on}}\quad \Gamma.
\end{equation}
On the nozzle walls $ \Gamma_w^+$  and $ \Gamma_w^- $, the usual slip boundary condition is imposed:
\begin{equation}\label{2-10}
  \frac{u_2^+}{u_1^+}=(g^+)^\prime(x_1), \ {\rm{on}}\ \Gamma_w^+, \quad
  \frac{u_2^-}{u_1^-}=(g^-)^\prime(x_1), \ {\rm{on}}\ \Gamma_w^-.
\end{equation}

 \par In summary, we will investigate the following problem:
 \begin{problem}
  Given  functions $ ( A_{en}^\pm, B_{en}^\pm,J_{en}^\pm)(x_2)$ at the entrance and function $ \omega_{ex}(x_2)$ at the exit, find a unique piecewise smooth subsonic solution $ (\bm{U}^+ ,\bm{U}^-) $  defined on $ \Omega^+ $ and $ \Omega^- $ respectively,  with the contact discontinuity $ \Gamma: x_2=g_{cd}(x_1) $ satisfying the Euler system \eqref{1-1} in the sense of weak solution and  the Rankine-Hugoniot conditions in \eqref{2-9} and the slip boundary conditions in \eqref{2-10}.
  \end{problem}
  \par
The main theorem of this paper can be stated as follows.
   \begin{theorem}
    Let $  \bm U_{0}^\pm=(A_b^\pm,B_b^\pm,J_b^\pm)$, where
\begin{equation*}
  A_b^\pm=\frac{P_b}{(\rho_b^\pm)^\gamma}, \quad B_b^\pm=\frac12|u_b^\pm|^2+\frac{\gamma P_b}{(\gamma-1)\rho_b^\pm}, \quad J_b^\pm=\rho_b^\pm u_b^\pm.
\end{equation*}
Then given functions $( A_{en}^\pm, B_{en}^\pm,J_{en}^\pm,\omega_{ex}) $, we
define
 \begin{equation}\label{2-11}
   \begin{aligned}
   \sigma(\bm U_{en}^+,\bm U_{en}^-,\omega_{ex},g^+,g^-):&=
   \|\bm U_{en}^+ -\bm U_0^+\|_{1,\alpha;[0,1]}+\|\bm U_{en}^- -\bm  U_0^-\|_{1,\alpha;[-1,0]} \\
   &\quad+\| \omega_{ex}\|_{2,\alpha;[g^-(L),g^+(L)]}+
   \| g^+ - 1\|_{1,\alpha;[0,L]}\\
   &\quad+
   \| g^- + 1\|_{1,\alpha;[0,L]}.
   \end{aligned}
   \end{equation}
  There exist positive constants $\sigma_{cd} $ and $ \mc $ depending only on  $ (\bm{U}_b^+,\bm{U}_b^+,L,\alpha) $ such that  if
   \begin{equation}\label{2-12}
   \begin{aligned}
   \sigma(\bm U_{en}^+,\bm U_{en}^-,\omega_{ex},g^+,g^-)\leq \sigma_{cd},
   \end{aligned}
   \end{equation}
  $ \mathbf{Problem \ 2.1}  $ has a unique piecewise smooth subsonic flow  $ (\bm{U}^+,\bm{U}^-) $   with the contact discontinuity $ \Gamma: x_2=g_{cd}(x_1) $ satisfying the following properties.
 \begin{enumerate}[\rm(i)]
\item The piecewise smooth subsonic flow  $ (\bm{U}^+,\bm{U}^-)\in ( C^{1,\alpha}(\Omega^+)\cap C^{\alpha}(\overline{\Omega^+}))\times (C^{1,\alpha}(\Omega^-)\cap C^{\alpha}(\overline{\Omega^-}))$  satisfies the following estimate:
 \begin{equation}\label{2-13}
\|\bm{U}^+ -\bm{U}_b^+\|_{C^{\alpha}(\overline{\Omega^+})}+\|\bm{U}^- -\bm{U}_b^-\|_{C^{\alpha}(\overline{\Omega^-})}\leq \mc\sigma(\bm U_{en}^+,\bm U_{en}^-,\omega_{ex},g^+,g^-).
\end{equation}
\item The contact discontinuity curve $ g_{cd}(x_1)\in C^{1,\alpha}([0,L])  $ satisfies
    $g_{cd}(0)=0 $. Furthermore, it holds that
\begin{equation}\label{2-14}
 \|g_{cd} \|_{1,\alpha;[0,L]}\leq \mc\sigma(\bm U_{en}^+,\bm U_{en}^-,\omega_{ex},g^+,g^-).
 \end{equation}
 \end{enumerate}
 \end{theorem}
 \section{The reformulation of the contact discontinuity problem }\noindent
\par In this section, we first employ the Lagrangian transformation to straighten  the contact discontinuity, and  reformulate the free boundary value    problem 2.1. Then  by introducing a stream function and solving the hyperbolic
equations for the entropy and the Bernoulli's quantity, the Euler system is reduced to a
nonlinear second-order  equation for the stream function.  Finally, we state the implicit function theorem and  the main steps to solve the free boundary problem 3.1.
\subsection{Reformulation by the Lagrangian transformation}\noindent
 \par  Let $ (\bm{U}^+(x_1,x_2),\bm{U}^-(x_1,x_2), g_{cd}(x_1)) $ be a solution to  Problem 2.1.
 Define
 \begin{equation}\label{3-1}
 m^-=\int_{-1}^{0}J_{en}^-(s)\de s \quad {\rm{and}}\quad
  m^+=\int_{0}^{1}J_{en}^+(s)\de s.
  \end{equation}
   Then for any $x_1\in( 0,L)$, it follows from the conservation of mass equation in \eqref{1-1} that
 \begin{equation}\label{3-2}
 \int_{g^-(x_1)}^{g_{cd}(x_1)}\rho^- u_{1}^-(x_1,s)\de s=m^-, \quad
 \int_{g_{cd}(x_1)}^{g^+(x_1)}\rho^+ u_{1}^+(x_1,s)\de s =m^+.
 \end{equation}
 \par Let
  \begin{equation}\label{3-3}
  y_2(x_1,x_2)= \int_{g_{cd}(x_1)}^{x_2}\rho u_{1}(x_1,s)\de s.
  \end{equation}
    It is easy to verify that
  \begin{equation*}
  \frac{\p y_2}{\p x_1}=-\rho u_2, \quad  \frac{\p y_2}{\p x_2}=\rho u_1.
  \end{equation*}
  Thus define the Lagrangian transformation  as
  \begin{equation}\label{3-4}
   y_1=x_1,\quad
     y_2=y_2(x_1,x_2).
    \end{equation}
       A direct computation gives
   \begin{equation*}
   \frac{\p(y_1,y_2)}{\p(x_1,x_2)}=\left|
\begin{matrix} 1& 0 \\ -\rho u_2 & \rho u_1\end{matrix}\right|=\rho u_1.
 \end{equation*}
 So  if $ (\rho^\pm,u_1^\pm, u_2^\pm,P^\pm) $ are close to
$ (\rho_b^\pm,u_b^\pm,0,P_b) $, we have $ \rho^\pm u_1^\pm \geq \mc_b> 0 $, where $ \mc_b $ depends only on the  background solutions.  Hence the  Lagrangian transformation is invertible.
\par Under this transformation,  the domain $ \Omega $ becomes
\begin{equation*}
\mn:=\{(y_1,y_2): 0<y_1<L,\ -m^-<y_2<m^+\}.
\end{equation*}
 The  lower wall $ \Gamma_w^-$ and  upper wall $ \Gamma_w^+$ of the nozzle are  transformed  into
\begin{equation*}
\begin{aligned}
\Sigma_w^-:=\{(y_1,y_2): 0<y_1<L,\ y_2=-m^-\}, \quad
  \Sigma_w^+:=\{(y_1,y_2): 0<y_1<L,\ y_2=m^+\}.
\end{aligned}
\end{equation*}
On  the contact discontinuity curve $ \Gamma $, we have
\begin{equation*}
y_2(x_1,g_{cd}(x_1))=\int_{g_{cd}(x_1)}^{g_{cd}(x_1)}\rho u_{1}(x_1,s)\de s=0.
\end{equation*}
Hence the free boundary $\Gamma $ becomes a straight line
\begin{equation}\label{3-5}
\Sigma:=\{(y_1,y_2):0<y_1<L,\ y_2=0\}.
\end{equation}
\par Define
\begin{equation}\label{3-6}
\mn^-:=\mn\cap \{-m^-<y_2<0\}, \quad \mn^+:=\mn\cap \{0<y_2<m^+\}.
\end{equation}
The entrance and exit  of $ \mn^\pm $ are defined as
\begin{equation*}
\begin{aligned}
\Sigma_0^+:=\{(y_1,y_2):  y_1=0,\ 0<y_2<m^+\}, \quad \Sigma_0^-:=\{(y_1,y_2):  y_1=0,\ -m^-<y_2<0\},\\
\Sigma_L^+:=\{(y_1,y_2):  y_1=L,\ 0<y_2<m^+\}, \quad \Sigma_L^-:=\{(y_1,y_2):  y_1=L,\ -m^-<y_2<0\}.
\end{aligned}
\end{equation*}
 Let
\begin{equation*}
{\bm{ U}}(y_1,y_2)=
  \begin{cases}
 {\bm{U}}^+(y_1,y_2):=( \rho^+,u_1^+, u_{2}^+, P^+)(\mathbf{x}(y_1,y_2)),\quad {\rm{in}}\quad   \mn^+,\\
  {\bm{U}}^-(y_1,y_2):=( \rho^-,u_1^-,u_{2}^-, P^-)(\mathbf{x}(y_1,y_2)),  \quad {\rm{in}}\quad  \mn^-.\\
  \end{cases}
 \end{equation*}
Then the system \eqref{1-1} in the new coordinates  can be rewritten as
\begin{equation}\label{3-7}
\begin{cases}
\p_{y_1}\left(\frac{1}{\rho  u_1}\right)-\p_{y_2}\left(\frac{ u_2}{  u_1}\right)
=0,\\
\p_{y_1}  u_2+\p_{y_2}P=0,\\
\p_{y_1}  u_1+\frac{1}{\rho  u_1}(\p_{y_1}-\rho  u_2\p_{y_2}) P=0,\\
\p_{y_1}  A=0.\\
\end{cases}
\end{equation}
The background solutions in the  Lagrangian coordinates are
\begin{equation}\label{3-8}
 {\bm{U}}_b=
  \begin{cases}
  {\bm{U}}_b^{+}:=(\rho_b^{+},u_b^+,0,P_b),  \quad  {\rm{in}}\quad \mn_b^+,\\
  {\bm{U}}_b^{-}:=(\rho_b^{-},u_b^-,0,P_b),  \quad {\rm{in}}\quad \mn_b^-,\\
  \end{cases}
 \end{equation}
  where
  \begin{equation*}
\mn_b^+:=\{(y_1,y_2):  0<y_1<L,\ 0<y_2<m_b^+\}, \quad
\mn_b^-:=\{(y_1,y_2): 0<y_1<L,\ -m_b^-<y_2<0\},
\end{equation*}
and $ m_b^\pm=\rho_b^\pm u_b^\pm $. Without loss of generality, we assume that $ m^\pm=m_b^\pm $.
\par
In the new coordinates,  the boundary data \eqref{2-6} at the entrance is given by
 \begin{equation}\label{3-9}
 {\bm {U}}_{en}(y_2)=
\begin{cases}
  {\bm {U}}_{en}^+(y_2)=(\e A_{en}^+,\e B_{en}^+,\e J_{en}^+)(y_2),\quad {\rm{on}}\quad \Sigma_0^+,\\
  {\bm {U}}_{en}^-(y_2)=(\e A_{en}^-,\e B_{en}^-,\e J_{en}^-)(y_2),\quad {\rm{on}}\quad \Sigma_0^-,\\
  \end{cases}
  \end{equation}
  where
  \begin{equation*}
  (\e A_{en}^+,\e B_{en}^+,\e J_{en}^+)(y_2)=( A_{en}^+, B_{en}^+,J_{en}^+)\left(\int_{0}^{y_2}\frac{1}{J_{en}^+}(s)\de s\right),
  \end{equation*}
  and
  \begin{equation*}
   (\e A_{en}^-,\e B_{en}^-,\e J_{en}^-)(y_2)=( A_{en}^-, B_{en}^-,J_{en}^-)\left(-\int_{y_2}^{0}\frac{1}{J_{en}^-}(s)\de s\right).
\end{equation*}
  The boundary condition \eqref{2-7} at the exit become
  \begin{equation}\label{3-10}
 \omega(L,y_2)=\begin{cases}
   \omega_{ex}(x_2^+(L,y_2)),\quad {\rm{on}}\quad \Sigma_L^+,\\
   \omega_{ex}(x_2^-(L,y_2)),\quad {\rm{on}}\quad \Sigma_L^-,\\
  \end{cases}
  \end{equation}
where
  \begin{equation*}
\begin{aligned}
x_2^+(L,y_2)=g_{cd}(L)+\int_{0}^{y_2}\frac{1}{\rho^+  u_{1}^+}(L,s)\de s, \quad
x_2^-(L,y_2)=g_{cd}(L)-\int_{y_2}^{0}\frac{1}{\rho^- u_{1}^-}(L,s)\de s.
\end{aligned}
\end{equation*}
  Note that the flow angle at the exit  becomes non-local and nonlinear in the Lagrangian coordinates.
 \par
 The Rankine-Hugoniot conditions in \eqref{2-9}  can be rewritten
\begin{equation}\label{3-11}
\frac{ u_2^+(y_1,0)}{ u_1^+(y_1,0)}=\frac{ u_2^-(y_1,0)}{ u_1^-(y_1,0)}=g_{cd}^{\prime}(y_1),
\end{equation}
and
\begin{equation}\label{3-12}
 P^+(y_1,0)=P^-(y_1,0).
\end{equation}
The slip boundary conditions in \eqref{2-10}  become
\begin{equation}\label{3-13}
 \frac{u_2^+}{ u_1^+}(y_1,m^+)=(g^+)^{\prime}(y_1), \ {\rm{on}}\ \Sigma_w^+, \quad \frac{ u_2^-}{u_1^-}(y_1,-m^-)=(g^-)^{\prime}(y_1), \ {\rm{on}}\ \Sigma_w^-.
\end{equation}
\subsection{The stream function formulation }\noindent
   \par By the first equation in \eqref{3-7},
  one can introduce a stream function
 \begin{equation}\label{3-14}
 \varphi(y_1,y_2)=
  \begin{cases}
  \varphi^-(y_1,y_2),\quad {\rm{for}}\quad (y_1,y_2)\in \mn^-,\\
    \varphi^+(y_1,y_2),\quad {\rm{for}}\quad (y_1,y_2)\in  \mn^+,\\
  \end{cases}
  \end{equation}
  such that
\begin{equation}\label{3-15}
\p_{y_1}\varphi=\frac{  u_2}{ u_1}, \quad \p_{y_2}\varphi=\frac{1}{\rho u_1}, \quad \varphi(0,0)=0.
\end{equation}
That is
\begin{equation*}
 u_1=\frac{1}{ \rho\p_{y_2}\varphi}, \quad  u_2=\frac{\p_{y_1}\varphi}{  \rho\p_{y_2}\varphi}.
\end{equation*}
\par It follows from  the fourth equation in \eqref{3-7} that one has
\begin{equation}\label{3-16}
A=\e A_{en}(y_2).
\end{equation}
Furthermore,  the Bernoulli's quantity  $ B $  in the Lagrangian coordinates satisfies
\begin{equation*}
 \p_{y_1} B=0.
 \end{equation*}
 Thus the following equation holds:
\begin{equation}\label{3-17}
\frac{1}{2}\left(\frac{1}{( \rho\p_{y_2}\varphi)^2}+\frac{(\p_{y_1}\varphi)^2}
{( \rho\p_{y_2}\varphi)^2}\right)+\frac{\e A_{en}\gamma \rho^{\gamma-1}}{\gamma  - 1}= \e B_{en}(y_2).
\end{equation}
Define
\begin{equation*}
\mm(\rho,\n \varphi,\e A_{en},\e B_{en})=\frac{1+(\p_{y_1}\varphi)^2}{2(\p_{y_2}\varphi)^2}+\frac{\e A_{en}\gamma \rho^{\gamma+1}}{\gamma- 1}-\e B_{en}\rho^{2}.
\end{equation*}
 Since the flow is subsonic, we have
\begin{equation*}
\frac{\p }{\p\rho}\mm(\rho,\n \varphi,\e A_{en},\e B_{en})=\rho(c^2(\rho, \e A_{en})-u_{1}^2- u_{2}^2)>0.
\end{equation*}
   By solving \eqref{3-17},  $ \rho $  as a function of  $ (\n \varphi, \e A_{en}, \e B_{en}) $, namely, $ \rho=\rho(\n \varphi, \e A_{en}, \e B_{en}) $.  Moreover, a direct computation yields that
\begin{equation*}
\begin{aligned}
\frac{\p \rho}{\p(\p_{y_1}\varphi)}&=-\frac{\p_{y_1}\varphi}{\rho(\p_{y_2}\varphi)^2
\left(c^2(\rho, \e A_{en})-\frac{1+(\p_{y_1}\varphi)^2}{( \rho\p_{y_2}\varphi)^2}\right)},\\
\frac{\p \rho}{\p(\p_{y_2}\varphi)}&=\frac{1+(\p_{y_1}\varphi)^2}
{\rho(\p_{y_2}\varphi)^3
\left(c^2(\rho, \e A_{en})-\frac{1+(\p_{y_1}\varphi)^2}{( \rho\p_{y_2}\varphi)^2}\right)},\\
\frac{\p \rho}{\p  \e A_{en}}&=-\frac{\gamma\rho^{\gamma}}{(\gamma-1)
\left(c^2(\rho, \e A_{en})-\frac{1+(\p_{y_1}\varphi)^2}{( \rho\p_{y_2}\varphi)^2}\right)},\\
\frac{\p \rho}{\p  \e B_{en}}&=\frac{\rho}{
c^2(\rho, \e A_{en})-\frac{1+(\p_{y_1}\varphi)^2}{( \rho\p_{y_2}\varphi)^2}}.\\
\end{aligned}
\end{equation*}
 Thus the solution $ \bm  U $ can be expressed as a vector-value function of  $ \varphi, \e A_{en}, \e B_{en} $:
\begin{equation}\label{3-18}
\bm U=( \rho, u_1, u_2, P)
=\left( \rho,\frac{1}{ \rho\p_{y_2}\varphi},\frac{\p_{y_1}\varphi}{\rho \p_{y_2}\varphi}, \e A_{en} \rho^{\gamma}\right).
\end{equation}
\par Denote
\begin{equation*}
 u_2=W_1(\n \varphi, \e A_{en}, \e B_{en}), \quad  P=\e A_{en}(\rho(\n \varphi, \e A_{en}, \e B_{en}))^\gamma=W_2(\n \varphi, \e A_{en}, \e B_{en}).
\end{equation*}
 Then the Euler system \eqref{3-7} is reduced to the following   second-order nonlinear equation for $ \varphi $:
  \begin{equation}\label{3-19}
    \p_{y_1}W_1(\n \varphi^+, \e A_{en}^+, \e B_{en}^+)+
     \p_{y_2}W_2(\n \varphi^+, \e A_{en}^+, \e B_{en}^+)=0, \quad {\rm{in}} \quad \mn^+,
 \end{equation}
 and
 \begin{equation}\label{3-20}
     \p_{y_1}W_1(\n \varphi^-, \e A_{en}^-, \e B_{en}^-)+
     \p_{y_2}W_2(\n \varphi^-, \e A_{en}^-, \e B_{en}^-)=0, \quad {\rm{in}} \quad \mn^-,
 \end{equation}
 with the following boundary conditions:
 \begin{equation}\label{3-21}
   \begin{cases}
 \varphi^+(0,y_2)=\int_{0}^{y_2}\frac{1}{\e J_{en}^+}(s)\de s,\quad &{\rm{on}} \quad \Sigma_0^+,\\
\p_{y_1}\varphi^+(L,y_2)=\omega_{ex}(\varphi^+(L,y_2)), \quad  &{\rm{on}} \quad \Sigma_L^+,\\
  \varphi^+(y_1,0)=g_{cd}(y_1),  \ &{\rm{on}} \quad  \Sigma,\\
 \varphi^+(y_1,m^+)=g^+(y_1)-1+\int_{0}^{m^+}\frac{1}{\e J_{en}^+}(s)\de s, \quad &{\rm{on}} \quad \Sigma_w^+,\\
 \end{cases}
 \end{equation}
 and
  \begin{equation}\label{3-22}
 \begin{cases}
 \varphi^-(0,y_2)=-\int_{y_2}^{0}\frac{1}{\e J_{en}^-}(s)\de s,\quad &{\rm{on}} \quad \Sigma_0^-,\\
 \p_{y_1}\varphi^-(L,y_2)=\omega_{ex}(\varphi^-(L,y_2)), \quad  &{\rm{on}} \quad \Sigma_L^-,\\
  \varphi^-(y_1,0)=g_{cd}(y_1),  \ &{\rm{on}} \quad  \Sigma,\\
 \varphi^-(y_1,-m^-)=g^-(y_1)+1-\int_{-m^-}^{0}\frac{1}{\e J_{en}^-}(s)\de s, \quad &{\rm{on}} \quad \Sigma_w^-.\\
 \end{cases}
 \end{equation}
  Here the flow angular at  the exit
  becomes local and nonlinear after introducing the stream function.
   \par  On the contact discontinuity,   \eqref{3-12} becomes
    \begin{equation}\label{3-23}
   W_2(\n \varphi^+,\e A_{en}^+,\e B_{en}^+ )(y_1,0)=
    W_2(\n \varphi^-,\e A_{en}^-,\e B_{en}^- )(y_1,0),
    \quad {\rm{on}} \quad \Sigma.\\
     \end{equation}
      \par Therefore, $ \mathbf{Problem \ 2.1}  $ is reformulated as follows.
 \begin{problem}
  Given  functions $ ( A_{en}^\pm, B_{en}^\pm,J_{en}^\pm)$ at the entrance and function $ \omega_{ex}$ at the exit, find a unique piecewise smooth subsonic solution $ (\varphi^+ ,\varphi^-;g_{cd}) $ separated by the straight line $ \Sigma $ such that the following  properties satisfied:
  \begin{enumerate}[\rm(i)]
\item
  $\varphi^+ $ and $ \varphi^-  $ are $ C^2 $ solutions to nonlinear second-order  equations \eqref{3-18} in $ \mn^+ $ and \eqref{3-19}  in $ \mn^- $, respectively;
  \item The   boundary conditions \eqref{3-21}-\eqref{3-23} are   satisfied.
      \end{enumerate}
 \end{problem}
 To prove Theorem 2.2, we first introduce some weight H$\ddot{\rm{o}}$lder spaces and their norms in rectangle $ \mn^\pm $. For $ \mathbf{y}=(y_1,y_2) $ and $ \e{\mathbf{y}}=(\e y_1,\e y_2)\in \mn^+ $, set
\begin{equation*}
 \delta_{\mathbf{y}}:=\min(y_2,m^+-y_2)\quad {\rm{and}} \quad
  \delta_{\textbf{y},\tilde{\textbf{y}}}:
  =\min(\delta_{\mathbf{y}},\delta_{\tilde{\mathbf{y}}}).
  \end{equation*}
  For any positive integer $ m $, $ \alpha\in(0,1) $ and $ \kappa\in \mathbb{R} $ and  a function $ u^+ $ defined on $ \mn^+ $, we define:
  \begin{equation*}
  \begin{aligned}
  {[u^+]}_{k,0;\mn^+}^{(\kappa)}\ &:=\sup_{|\beta|=k} \delta_{\mathbf{y}}^{\max(|\beta|+\kappa,0)}|D^{\beta}u^+(\textbf{y})|, \ k=0,1,\cdots,m;\\
  [u^+]_{m,\alpha;\mn^+}^{(\kappa)}\ &:=\sup_{|\beta|=m}\delta_{\mathbf{y},\tilde{\mathbf{y}}}^{\max(m+\alpha+\kappa,0)}
  \frac{|D^{\beta}u^+(\textbf{y})-D^{\beta}u^+(\tilde{\textbf{y}})|}{|\textbf{y}-\tilde{\textbf{y}} |^{\alpha}};\\
  \|u^+\|_{m,\alpha;\mn^+}^{(\kappa)}\ &:=\sum_{k=0}^{m}[u^+]_{k,0,\mn^+}^{(\kappa)}
  +[u^+]_{m,\alpha;\mn^+}^{(\kappa)}\\
  \end{aligned}
  \end{equation*}
  with the corresponding function space defined as
  \begin{equation*}
  C_{m,\alpha}^{(\kappa)}(\mn^+)=\{u^+:\|u^+\|_{m,\alpha;\mn^+}^{(\kappa)}<\infty\}.
  \end{equation*}
 \par  Similarly,  for $ \mathbf{y}=(y_1,y_2) $ and $ \e{\mathbf{y}}=(\e y_1,\e y_2)\in \mn^- $, set
\begin{equation*}
 \delta_{\mathbf{y}}:=\min(-y_2,m^-+y_2)\quad {\rm{and}} \quad
  \delta_{\textbf{y},\tilde{\textbf{y}}}:
  =\min(\delta_{\mathbf{y}},\delta_{\tilde{\mathbf{y}}}).
  \end{equation*}
Then for a function $ u^- $ defined on $ \mn^- $, one can define $   {[u^-]}_{k,0;\mn^-}^{(\kappa)} $,  $ [u^-]_{m,\alpha;\mn^-}^{(\kappa)} $, $ \|u^-\|_{m,\alpha;\mn^-}^{(\kappa)} $ and $  C_{m,\alpha}^{(\kappa)}(\mn^-) $, respectively.

\par  Theorem 2.2  then follows directly from  the following theorem:
 \begin{theorem}
     Let  $ \varphi _b^\pm(y_2)=\frac{y_2}{m^\pm} $. Then
 there exist positive constants $\sigma_{cd}^\ast $ and $ \mc^\ast $ depending only on  $ ({\bm{U}}_b^+,{\bm{U}}_b^-,L,\alpha) $ such that  if
   \begin{equation}\label{3-24}
   \begin{aligned}
  \sigma(\bm U_{en}^+,\bm U_{en}^-,\omega_{ex},g^+,g^-)\leq  \sigma_{cd}^\ast ,
   \end{aligned}
   \end{equation}
   $ \mathbf{Problem \ 3.1}  $ has a unique piecewise smooth subsonic solution $ (\varphi^+,\varphi^-;g_{cd}(y_1))$ satisfying properties $ \rm(i) $ and $ \rm(ii) $.  Furthermore, the following estimate holds:
    \begin{equation}\label{3-27}
  \|\varphi^+ -\varphi_b^+\|_{2,\alpha;\mn^+}^{(-1-\alpha)}
    +\|\varphi^- -\varphi_b^-\|_{2,\alpha;\mn^-}^{(-1-\alpha)}+ \|g_{cd} \|_{1,\alpha;[0,L]}
   \leq\mc^\ast \sigma(\bm U_{en}^+,\bm U_{en}^-,\omega_{ex},g^+,g^-).
   \end{equation}
 \end{theorem}
 \subsection{Solving the free boundary problem 3.1}\noindent
\par    Inspired by  \cite{CXZ22}, we  use the implicit function theorem to solve this free boundary problem. To this end, we first state the implicit function theorem.
\begin{theorem}
   (Implicit function theorem \cite[Theorem 4.B]{ZE92}). Let $ X $, $Y$, $Z $ be Banach spaces and let $ U(x_0,y_0) $ be an open neighborhood of $ (x_0,y_0) $ in $ X\times Y $. Suppose that the mapping $ F:U(x_0,y_0)\rightarrow Z $ satisfies the following conditions:
   \begin{enumerate}[\rm(i)]
\item $ F(x_0,y_0)=0 $;
  \item The partial Fr\'{e}chet derivative $ D_yF $ exists on $ U(x_0,y_0) $ and $ D_yF(x_0,y_0)$: $Y\rightarrow Z  $ is bijective;
  \item $ F $ and $ D_yF $ are continuous at $ (x_0,y_0) $.
      \end{enumerate}
      Then the following hold:
    \begin{enumerate}[\rm(1)]
 \item Existence and uniqueness. There exist positive numbers $ r_0 $ and $ r $ such that,
for every $ x \in X $ satisfying $ \| x - x_0 \| < r_0 $,  $F(x,y) = 0  $ has a unique solution  $ y(x) \in Y $  with
 $ \|y(x) - y_0 \|<r $.
 \item  Continuity. If $ F $ is continuous in a neighborhood of $(x_0, y_0) $, then $ y(\cdot) $ is
continuous in a neighborhood of $ x_0 $.
\item Continuous differentiability. If $ F $ is a $ C^m $-map, $ 1 \leq m < \infty $, on a neighborhood of $ (x_0, y_0) $, then $ y(\cdot) $ is also a $ C^m $-map on a neighborhood of $x_0 $.
     \end{enumerate}
     \end{theorem}
     \par Then we follow the steps below to solve Problem 3.1:
     \begin{enumerate}[ \bf (a)]
 \item
 Given any function $ g_{cd}(y_1)=\int_{0}^{y_1}\eta(s)\de s $ belonging to some suitable function classes, we  will solve the nonlinear second-order equation \eqref{3-19}(\eqref{3-20}) with nonlinear and mixed boundary condition \eqref{3-21}(\eqref{3-22}) in $ \mn^+ $( in $ \mn^- $). This will be achieved by constructing suitable barrier functions and employing Schauder estimates for second-order elliptic equation. The detailed analysis will be given in Section 4.
 \item Define the map $  \mq(\bm \zeta_0,\eta):=(W_2(\n \varphi^+,\e A_{en}^+,\e B_{en}^+ )-  W_2(\n \varphi^-,\e A_{en}^-,\e B_{en}^- ))(y_1,0) $. To employ the implicit function theorem, we need to  compute   the Fr\'{e}chet derivative $ D_\eta\mq(\bm \zeta_0,\eta)$ of the functional $ \mq(\bm \zeta_0,\eta)$ with respect to $ \eta $ and show that $D_\eta\mq(\bm\zeta_b,0) $ is an  isomorphism. This step will be achieved in  Section 5.
           \end{enumerate}
 \section{The  solution to a fixed boundary value problem in $\mn $}\noindent
\par In this section,  given any function $ g_{cd}(y_1)=\int_{0}^{y_1}\eta(s)\de s \in C^{1,\alpha}([0,L]) $,
 we  will solve the nonlinear second-order equation \eqref{3-19}(\eqref{3-20}) with nonlinear and mixed boundary condition \eqref{3-21}(\eqref{3-22}) in $ \mn^+ $( in $ \mn^- $). For convenience, we will focus on $ \mn^+ $.
 \subsection{Linearization}\noindent
\par To solve nonlinear equation  \eqref{3-19} in the  domain $ \mn^+ $, we first linearize \eqref{3-19} and then solve
the linear equation in the  domain $ \mn^+ $.
\par Let $ \phi^+=\varphi^+-\varphi_b^+ $. Note that
 \begin{equation}\label{4-1}
\p_{y_1}W_1(\n \varphi_b^+, A_b^+, B_b^+)+
     \p_{y_2}W_2(\n \varphi_b^+,  A_b^+,  B_b^+)=0, \quad {\rm{in}} \quad \mn^+.
     \end{equation}
     Taking the difference of \eqref{3-19} and \eqref{4-1}, one has
     \begin{equation}\label{4-2}
\sum_{i,j=1,2}\p_{y_i}(a_{ij}(\n \phi^+)\p_{y_j}\phi^+)=\sum_{i=1,2}\p_{y_i}f_i^+, \quad {\rm{in}} \quad \mn^+,
     \end{equation}
     where
     \begin{equation*}
     \begin{aligned}
     a_{ij}(\n \phi^+)=\int_0^1 \p_{ \varphi^+_{y_j}}
     W_i(\n\varphi_b^++s  \n \phi^+, \e A_{en}^+, \e B_{en}^+)\de s, \quad i,j=1,2,
     \end{aligned}
     \end{equation*}
     and
     \begin{equation*}
     \begin{aligned}
     f_i^+=W_i(\n\varphi_b^+, A_b^+,  B_b^+)-W_i(\n\varphi_b^+, \e A_{en}^+, \e B_{en}^+) , \quad i=1,2.
     \end{aligned}
     \end{equation*}
     \par   Define an iteration set
\begin{equation}\label{4-3}
\mj_{upp}(\delta_1)=\{\phi^+:\phi^+(0,0)=0, \|\phi^+\|_{2,\alpha;\mn^+}^{(-1-\alpha)}
\leq
\delta_1\},
\end{equation}
where $ \delta_1 $ is a positive constant to be determined later. Then for given $ \bar\varphi^+ $ such that $ \bar\phi^+=\bar\varphi^+-\varphi^+_b\in \mj_{upp}(\delta_1) $, find $ \varphi^+=\varphi^+_b+  \phi^+ $ by  solving the  linear equation:
 \begin{equation}\label{4-4}
 \begin{cases}
\sum_{i,j=1,2}\p_{y_i}(a_{ij}(\n \bar\phi^+)\p_{y_j}\phi^+)=\sum_{i=1,2}\p_{y_i}f_i^+, &\quad {\rm{in}} \quad \mn^+,\\
    \phi^+(0,y_2)=g_{1}^+(y_2), &\quad {\rm{on}} \quad
  \Sigma_0^+,\\
 \phi^+(y_1,0)=g_{cd}(y_1), &\quad {\rm{on}} \quad
  \Sigma,\\
  \phi^+(y_1,m^+)=g_{2}^+(y_1), &\quad {\rm{on}} \quad
  \Sigma_w^+,\\
  \p_{y_1}\phi^+(L,y_2)=\e\omega_{ex}^+(y_2), &\quad {\rm{on}} \quad
  \Sigma_L^+,\\
  \end{cases}
  \end{equation}
  where
  \begin{equation*}
     \begin{aligned}
    g_{1}^+(y_2)&=\int_0^{y_2}\left(\frac{1}{\e J_{en}^+(s)}-\frac{1}{ J_b^+}\right)\de s,\\
    g_{2}^+(y_1)&=g^+(y_1)-1+\int_0^{m^+}\left(\frac{1}{\e J_{en}^+(s)}-\frac{1}{ J_b^+}\right)\de s,\\
    \e\omega_{ex}^+(y_2)&= \omega_{ex}(\varphi_b^+(y_2)+\bar\phi^+(L,y_2)).
    \end{aligned}
    \end{equation*}
    \par Use the abbreviation
    \begin{equation*}
    \sigma_v=\sigma(\bm U_{en}^+,\bm U_{en}^-,\omega_{ex},g^+,g^-),
    \end{equation*}
    where $\sigma(\bm U_{en}^+,\bm U_{en}^-,\omega_{ex},g^+,g^-)$  is  defined in \eqref{2-11}.
 Then a direct computation yields that
\begin{equation}\label{4-6}
\begin{cases}
\begin{aligned}
&\sum_{i=1}^2\| f_i^+\|_{1,\alpha;\mn^+}^{(-\alpha)}
\leq\mc^+ \sigma_v,\\
  &\|g_{1}^+\|_{2,\alpha;\Sigma_0^+}^{(-1-\alpha)}
\leq \mc^+ \|\e J_{en}^+ -J_b^+\|_{1,\alpha;\Sigma_0^+}^{(-\alpha)}
\leq\mc^+\sigma_v,\\
&\|g_{2}^+\|_{1,\alpha;
[0,L]}
\leq \mc^+\left(
\|g^+(y_1)-1\|_{1,\alpha;
[0,L]}+\|\e J_{en}^+ -J_b^+\|_{1,\alpha;\Sigma_0^+}^{(-\alpha)}\right)\leq\mc^+\sigma_v,\\
&\|\e\omega_{ex}^+\|_{1,\alpha;\overline{\Sigma_L^+}}\leq \mc^+(\sigma_v\delta_1+ \sigma_v),
\end{aligned}
\end{cases}
  \end{equation}
  where $ \mc^+>0 $  depends only on $ \bm U_b^+ $, $L$ and $ \alpha $.
   \subsection{Solving the linear boundary value problem  }\noindent
  \par In this subsection, we consider the following linear boundary value problem:
  \begin{equation}\label{4-7}
  \begin{cases}
\sum_{i,j=1,2}\p_{y_i}(a_{ij}(\n \bar\phi^+)\p_{y_j}\phi^+)=\sum_{i=1}^2\p_{y_i}f_i^+, &\quad {\rm{in}} \quad \mn^+,\\
\phi^+(0,y_2)=g_{1}^+(y_2), &\quad {\rm{on}} \quad
  \Sigma_0^+,\\
 \phi^+(y_1,0)=g_{cd}(y_1), &\quad {\rm{on}} \quad
  \Sigma,\\
  \phi^+(y_1,m^+)=g_{2}^+(y_1), &\quad {\rm{on}} \quad
  \Sigma_w^+,\\
  \p_{y_1}\phi^+(L,y_2)=\e\omega_{ex}^+(y_2), &\quad {\rm{on}} \quad
  \Sigma_L^+.\\
  \end{cases}
      \end{equation}
      For the coefficients of \eqref{4-7}, it is easy to see that
\begin{equation}\label{4-8}
\|a_{ij}-e_i^+\delta_{ij}\|_{1,\alpha;\mn^+}^{(-\alpha)}\leq \kappa^+,
\end{equation}
  where $\kappa^+>0$  depends on $ \delta_1 $,  $ \bm U_b^+ $, $ L $ and $ \alpha $. Note that
   \begin{equation}\label{4-9}
 e_{1}^+=a_{11}(0 )=u_b^+>0,\quad
 e_{2}^+=a_{22}(0 )=\frac{  (c_b^+)^2(u_b^+)^3(\rho_b^+)^2}
{(c_b^+)^2-(u_b^+)^2}>0,
\end{equation}
and  $ a_{12}(0)=a_{21}(0)=0 $. Hence one has
\begin{equation*}
a_{11}(0 )a_{22}(0)-(a_{12}(0))^2= \frac{(c_b^+)^2(\rho_b^+)^2(u_b^+)^4}
   {(c_b^+)^2-(u_b^+)^2}>0.
   \end{equation*}
    If  $\kappa^+  $ is  sufficiently
small, there exists a constant $ \lambda^+>0 $ depending only on $ {\bm U}_b^+ $  such that for any $ \bm \xi \in \mathbb{R}^2 $, it holds that
\begin{equation}\label{4-10}
 \lambda^+ |\bm \xi|^2\leq a_{ij}\xi_i\xi_j\leq \frac{1}{\lambda^+}|\bm \xi|^2.
 \end{equation}
 \par Then for the   problem \eqref{4-7}, we have the following conclusion:
   \begin{lemma}
   For given $  f_i^+\in C_{1,\alpha}^{(-\alpha)}(\mn^+) $ $(i=1,2) $, $  g_1^+\in C_{2,\alpha}^{(-1-\alpha)}(\Sigma_0^+) $, $ g_{cd}\in C^{1,\alpha}([0,L]) $, $ g_{2}\in C^{1,\alpha}([0,L]) $   and $\e\omega_{ex}^+\in  C^{1,\alpha}(\overline{\Sigma_L^+}) $,   the boundary value problem \eqref{4-7} has  a unique solution $ \phi^+ \in C_{2,\alpha}^{(-1-\alpha)}(\mn^+) $ satisfying
\begin{equation}\label{4-11}
\begin{aligned}
\|\phi^+\|_{2,\alpha;\mn^+}^{(-1-\alpha)}
\leq \mc^+\bigg(\sum_{i=1}^2\| f_i^+\|_{1,\alpha;\mn^+}^{(-\alpha)}+\|g_{1}^+\|_{2,\alpha;\Sigma_0^+}^{(-1-\alpha)}
+\|g_{cd}\|_{1,\alpha;
[0,L]}
+\|g_{2}^+\|_{1,\alpha;
[0,L]}
+\|\e\omega_{ex}^+\|_{1,\alpha;\overline{\Sigma_L^+}}\bigg),\\
\end{aligned}
\end{equation}
where the constant $ \mc^+>0 $ depends on  $ \bm U_b^+ $, $L$ and $ \alpha $.
  \end{lemma}
 \par In order to prove Lemma 4.1, we first  consider an auxiliary problem:
  \begin{equation}\label{4-12}
  \begin{cases}
e_1^+\p_{y_1}^2\phi^++e_2^+\p_{y_2}^2\phi^+
=\sum_{i=1}^2\p_{y_i}h_i^+, &\quad {\rm{in}} \quad \mn^+,\\
\phi^+(0,y_2)=g_{1}^+(y_2), &\quad {\rm{on}} \quad
  \Sigma_0^+,\\
 \phi^+(y_1,0)=g_{cd}(y_1), &\quad {\rm{on}} \quad
  \Sigma,\\
 \phi^+(y_1,m^+)=g_{2}^+(y_1), &\quad {\rm{on}} \quad
  \Sigma_w^+,\\
  \p_{y_1}\phi^+(L,y_2)=\e\omega_{ex}^+(y_2), &\quad {\rm{on}} \quad
  \Sigma_L^+.\\
  \end{cases}
      \end{equation}
      \begin{lemma}
  For given $  h_i\in C_{1,\alpha}^{(-\alpha)}(\mn^+) $ $(i=1,2) $, $  g_1^+\in C_{2,\alpha}^{(-1-\alpha)}(\Sigma_0^+) $, $ g_{cd}\in C^{1,\alpha}([0,L]) $, $ g_{2}\in C^{1,\alpha}([0,L]) $   and $\e\omega_{ex}^+\in  C^{1,\alpha}(\overline{\Sigma_L^+}) $,
   the boundary value problem \eqref{4-12} has a unique solution $ \phi^+ \in C_{2,\alpha}^{(-1-\alpha)}(\mn^+) $ such that
\begin{equation}\label{4-13}
\begin{aligned}
\|\phi^+\|_{2,\alpha;\mn^+}^{(-1-\alpha)}
\leq \mc^+\bigg(\sum_{i=1}^2\| h_i^+\|_{1,\alpha;\mn^+}^{(-\alpha)}+\|g_{1}^+\|_{2,\alpha;\Sigma_0^+}^{(-1-\alpha)}
+\|g_{cd}\|_{1,\alpha;
[0,L]}
+\|g_{2}^+\|_{1,\alpha;
[0,L]}
+\|\e\omega_{ex}^+\|_{1,\alpha;\overline{\Sigma_L^+}}\bigg),
\end{aligned}
\end{equation}
where $ \mc^+>0 $ depends on  $ \bm U_b^+ $, $L$ and $ \alpha $.
\end{lemma}
\begin{proof}
By \cite[Theorem 1]{LG86}, the boundary value problem \eqref{4-12} has a unique solution
\begin{equation*}
\phi^+\in C^{2,\alpha}(\mn^+)\cap C^0(\overline{\mn^+}).
\end{equation*}
 Next, we obtain the  estimate of $\phi^+ $  in the weighted norm $ C_{2,\alpha;\mn^+}^{(-1-\alpha)} $. By employing the transformation $ (\tilde y_1,\tilde y_2)=\left(\frac{y_1}{\sqrt{e_1^+}},\frac{y_2}{\sqrt{e_2^+}}\right)$, the operator $ e_1^+\p_{y_1}^2+e_2^+\p_{y_2}^2 $ is transformed into the Laplace operator. Thus, for convenience, below we will replace \eqref{4-12}  by
  \begin{equation}\label{4-14}
  \begin{cases}
\Delta \phi^+
=\sum_{i=1}^2\p_{y_i}h_i^+, &\quad {\rm{in}} \quad \mn^+,\\
\phi^+(0,y_2)=g_{1}^+(y_2), &\quad {\rm{on}} \quad
  \Sigma_0^+,\\
 \phi^+(y_1,0)=g_{cd}(y_1), &\quad {\rm{on}} \quad
  \Sigma,\\
 \phi^+(y_1,m^+)=g_{2}^+(y_1), &\quad {\rm{on}} \quad
  \Sigma_w^+,\\
  \p_{y_1}\phi^+(L,y_2)=\e\omega_{ex}^+(y_2), &\quad {\rm{on}} \quad
  \Sigma_L^+.\\
  \end{cases}
      \end{equation}
 We will divide into the following five steps   to derive the estimate of the solution $ \phi^+$ for the boundary value problem \eqref{4-14}.
\par  { \bf Step 1. Homogenize the boundary data.}
\par  We first seek a function $ \phi^+_1\in  C_{2,\alpha}^{(-1-\alpha)}(\mn^+) $ satisfying $ \p_{y_1}\phi^+_1(L,y_2)=\e\omega_{ex}^+(y_2) $. It follows from Lemma 6.38 in \cite{GT98} that $ \e\omega_{ex}^+$ can be extended outside $ \Sigma_{L}^+ $ so that $ \e\omega_{ex}^+\in C_0^{1,\alpha}(\mathbb{R}) $. Let $ \xi(t) $ be a smooth mollifier satisfying $ \xi(t)\geq 0 $ and  $ \int_{-\infty}^{\infty}\xi(t)\de t=1 $. We define
\begin{equation*}
  \phi^+_1(y_1,y_2)=(y_1-L)\int_{-\infty}^{\infty}
 \e\omega_{ex}^+(y_2-(y_1-L)t)\xi(t)\de t.
  \end{equation*}
  Then one has $ \p_{y_1}\phi_1^+(L,y_2)=\e\omega_{ex}^+(y_2) $ and
  \begin{equation}\label{4-15}
  \|\phi^+_1\|_{2,\alpha;\mn^+}^{(-1-\alpha)}\leq C\|\e\omega_{ex}^+\|_{1,\alpha;\overline{\Sigma_L^+}}.
  \end{equation}
  Set
  \begin{equation*}
   \phi^+_2(y_1,y_2)=\phi^+(y_1,y_2)-(\phi^+_1(y_1,y_2)-\phi^+_1(0,y_2))-g_{1}^+(y_2).
   \end{equation*}
   Then $ \phi^+_2$ satisfies
  \begin{equation}\label{4-16}
  \begin{cases}
\Delta \phi^+_2
=\sum_{i=1}^2\p_{y_i}h_{i+2}^+, &\quad {\rm{in}} \quad \mn^+,\\
\phi^+_2(0,y_2)=0, &\quad {\rm{on}} \quad
  \Sigma_0^+,\\
 \phi^+_2(y_1,0)=h_5^+(y_1), &\quad {\rm{on}} \quad
  \Sigma,\\
 \phi^+_2(y_1,m^+)=h_6^+(y_1), &\quad {\rm{on}} \quad
  \Sigma_w^+,\\
  \p_{y_1}\phi^+_2(L,y_2)=0, &\quad {\rm{on}} \quad
  \Sigma_L^+,\\
  \end{cases}
      \end{equation}
      where
   \begin{equation*}
  \begin{aligned}
  &h_3^+(y_1,y_2)=h_1^+(y_1,y_2)-\p_{y_1}\phi^+_1(y_1,y_2), \\    &h_4^+(y_1,y_2)=h_2^+(y_1,y_2)-(\p_{y_2}\phi^+_1(y_1,y_2)-
  \p_{y_2}\phi^+_1(0,y_2))-(g_1^+)^\prime(y_2),\\
  &h_5^+(y_1)=g_{cd}(y_1)-(\phi^+_1(y_1,0)-\phi^+_1(0,0))-g_1^+(0),\\
  &h_6^+(y_1)=g_{2}^+(y_1)-(\phi^+_1(y_1,m^+)-\phi^+_1(0,m^+))-g_1^+(m).
  \end{aligned}
      \end{equation*}

      \par  {\bf Step 2. The $ L^\infty  $-estimate of $\phi^+_2 $.}
\par To obtain \eqref{4-13}, we first need to the $ L^\infty  $-estimate of $\phi^+_2 $. To achieve this, define the following comparison function:
 \begin{equation*}
\begin{aligned}
v_1
=M\bigg
(1+\bigg(\frac{3L}{2}-y_1\bigg)^2+M_1(y_2^\alpha+(m^+-y_2)^\alpha)\bigg),
\end{aligned}
\end{equation*}
where
\begin{equation*}
M=\sum_{i=1}^2\|h_{i+2}\|_{1,\alpha;\mn^+}^{(-\alpha)}
+\sum_{i=5}^6\|h_{i}\|_{1,\alpha;[0,L]}, \quad {\rm{and}} \quad  M_1=\frac{4}{\alpha(1-\alpha)}(m^+)^{2-\alpha}.
\end{equation*}
\par A direct  computation yields that
\begin{equation*}
\begin{aligned}
\Delta v_1&= -MM_1\alpha(1-\alpha)
(y_2^{\alpha-2}+(m^+-y_2)^{\alpha-2})+2M\\
 &=-MM_1\alpha(1-\alpha)
(y_2^{\alpha-2}+(m^+-y_2)^{\alpha-2})
+\frac{\alpha(1-\alpha)}{2}MM_1(m^+)^{\alpha-2}\\
&\leq -MM_1\alpha(1-\alpha)
(y_2^{\alpha-2}+(m^+-y_2)^{\alpha-2})
+\frac{\alpha(1-\alpha)}{2}MM_1(y_2^{\alpha-2}+(m^+-y_2)^{\alpha-2})\\
 &\leq-\frac{1}{2m^+}MM_1\alpha(1-\alpha)
(y_2^{\alpha-1}+(m^+-y_2)^{\alpha-1}).\\
\end{aligned}
\end{equation*}
On the boundary $ \Sigma_{L}^+$, one derives that
\begin{equation*}
\p_{y_1}v_1(L,y_2)=-{ML}.
\end{equation*}
 Thus, for sufficiently large positive constant $ C_1 $ independent of $ M $,  we have
\begin{equation*}
\begin{cases}
\Delta(C_1v_1\pm \phi^+_2)\leq 0,\quad &{\rm{in}} \quad
\mn^+,\\
(C_1v_1\pm \phi^+_2)(0,y_2)>0,\quad &{\rm{on}} \quad \Sigma_0^+,\\
(C_1v_1\pm \phi^+_2)(y_1,0)>0,
 &\ {\rm{on}} \quad
\Sigma,\\
(C_1v_1\pm \phi^+_2)(y_1,m^+)>0,  &{\rm{on}} \quad
\Sigma_w^+,\\
\p_{y_1}(C_1v_1\pm \phi^+_2)(L,y_2)< 0 , \ &{\rm{on}} \quad
\Sigma_L^+.\\
\end{cases}
\end{equation*}
Then it follows from the comparison principle that

\begin{equation}\label{4-18}
\|\phi^+_2\|_{L^\infty}\leq C_1M, \quad {\rm{in }} \quad \mn^+.
\end{equation}
\par {\bf Step 3. The $ C^{1,\alpha} $  estimate of $ \phi_2^+ $ near the left corner points $ (0,0) $ and $ (0,m^+) $.}
\par In this step, we  focus on  the $ C^{1,\alpha} $  estimate of $\phi_2^+ $ near the corner point $ (0,0) $. The  corner point $ (0,m^+) $ can be treated similarly.
\par Define a comparison function $ v_2 $ in polar coordinates as follows
\begin{equation}\label{4-19}
v_2( r,\theta)=\frac{2}{\sin(\frac{1-\alpha}{5})} r^{1+\alpha}\sin(z( \theta))-r^{1+\alpha}\sin^{1+\alpha} ( \theta),\quad 0\leq \theta\leq \frac{\pi}{2},
\end{equation}
where
$  r=\sqrt{y_1^2+y_2^2} $, $ \theta=\arctan\frac{y_2}{y_1} $  and
\begin{equation*}
 z( \theta)=\frac{1-\alpha}{5}+\frac{6+4\alpha}{5} \theta, \quad \theta\in[0,\pi/2].
\end{equation*}
Note that $\frac{1-\alpha}{5}\leq z( \theta)\leq z\left(\frac{\pi}{2}\right)$ for $ \theta\in[0,\pi/2]$, and
\begin{equation*}
z\left(\frac{\pi}{2}\right)=
\frac{1-\alpha}{5}+\frac{6+4\alpha}{5}\frac{\pi}{2}=\pi-\frac{1-\alpha}{5}
-\frac{2(1-\alpha)}{5}(\pi-1)<\pi-\frac{1-\alpha}{5}.
\end{equation*}
 Thus $ \sin(z( \theta))\geq \sin(\frac{1-\alpha}{5}) $ and
\begin{equation}\label{4-20}
v_2(r, \theta)\geq r^{1+\alpha}\left(\frac{2\sin(\frac{1-\alpha}{5})}
{\sin(\frac{1-\alpha}{5})}-1\right)
\geq r^{1+\alpha}.
\end{equation}
Therefore
\begin{equation}\label{4-21}
\begin{aligned}
\Delta v_2&= (\p_{{r}}^2 + \frac{1}{{r}}\p_{{r}}+ \frac1{{r}^2}\p_{{\theta}}^2) v_2\\
&=\frac{2((5+5\alpha)^2-(6+4\alpha)^2)}{25
\sin(\frac{1-\alpha}{5})} r^{\alpha-1}\sin(z(\theta))-\alpha(1+\alpha) r^{\alpha-1}\sin^{\alpha-1} ( \theta)\\
&\leq -\alpha(1+\alpha) r^{\alpha-1}\sin^{\alpha-1} ( \theta)=-\alpha(1+\alpha)y_2^{\alpha-1}.\\
\end{aligned}
\end{equation}
\par Define
 \begin{equation*}
 \hat \phi_2^+= \phi_2^+-\phi_2^+(0,0)-\n \phi_2^+(0,0)\cdot  \mathbf{y}=\phi_2^+-
 (h_5^+)^\prime(0)y_1.
 \end{equation*}
 On the boundary $ \Sigma $, the following inequality holds:
 \begin{equation}\label{4-z}
 \hat h_5^+(y_1)=h_5^+(y_1)- h_5^+(0)-(h_5^+)^\prime(0)y_1\leq \|h_5^+\|_{1,\alpha;[0,L]}y_1^{1+\alpha}.
 \end{equation}
Let  $  r_0>0 $ be small and fixed and set $ B_{ r_0}^+=B_{ r_0}(0,0)\cap \mn^+ $. Then it follows from
\eqref{4-18}-\eqref{4-z} that there exists a suitably large positive constant  $ C_2$ independent of $ M$   such that
\begin{equation*}
\begin{cases}
\Delta(C_2Mv_2\pm \hat\phi_2^+)\leq 0, &\quad {\rm{in}} \quad B_{ r_0}^+,\\
C_2Mv_2\pm \hat\phi_2^+\geq0, &\quad {\rm{on}} \quad  B_{ r_0}^+\cap
\p \mn^+,\\
C_2Mv_2\pm \hat\phi_2^+\geq0,&\quad {\rm{on}} \quad  B_{ r_0}^+\cap \{r=r_0\}.\\
\end{cases}
  \end{equation*}
   This, together with the comparison principle, yields
\begin{equation}\label{4-22}
\|\hat\phi_2^+\|_{L^\infty}\leq C_2 r^{1+\alpha}
M, \quad {\rm{in }} \quad B_{r_0}^+.
\end{equation}
With estimate \eqref{4-22}, one can use the scaling technique to obtain the $ C^{1,\alpha} $ estimate up to the corner $ (0,0) $. More precisely, for any point $ P_\ast\in  B_{ r_0/2}^+ $ with polar coordinates $ (d_\ast,\theta_\ast) $, we consider the interior estimate and boundary  estimate for different values of $ \theta_\ast $.
\par  Case 1: Interior estimate for $ \theta_\ast\in[\frac{\pi}{6},\frac{\pi}{3}] $. Set $ B_1=B_{\frac{d_\ast}{6}}(P_\ast) $ and $ B_2=B_{\frac{d_\ast}{3}}(P_\ast) $, then  $ B_1\subset B_2 \subset B_{r_0}^+ $. By the Schauder interior estimate  \cite[Theorem 4.15]{GT98}, one has
\begin{equation}\label{4-y}
\|\hat \phi_2^+\|_{1,\alpha;B_1}^{(0)}\leq C\bigg(\|\hat \phi_2^+\|_{0,0;B_2}+
\sum_{i=1}^2\|h_{i+2}\|_{0,\alpha;B_2}^{(1)}\bigg),
\end{equation}
where    $ C $ is a positive constant independent of $ d_\ast $ and the weighted H\"{o}lder norm is defined as
\begin{equation*}
\|u\|_{k,\alpha;\md}^{(l)}=\sum_{j=0}^k d^{j+l}\sup_{|\beta|=j}\sup_{\mathbf{x}\in\md}|D^\beta u(\mathbf{x})|+d^{j+\alpha+l}\sup_{|\beta|=k}\sup_{\mathbf{x}_1,\mathbf{x}_2\in\md}
\frac{|D^\beta u(\mathbf{x}_1)-D^\beta u(\mathbf{x}_2)|}{|\mathbf{x}_1-\mathbf{x}_2|^\alpha}, \ d={\rm{diam}}\md.
\end{equation*}
   The estimate \eqref{4-y},  together with  \eqref{4-22}, yields that
\begin{equation}\label{4-x}
\|\hat \phi_2^+\|_{1,\alpha;B_1} \leq Cd_{\ast}^{-1-\alpha}\|\hat \phi_2\|_{1,\alpha;B_1}^{(0)}
\leq CM.
\end{equation}
\par  Case 2: Boundary  estimate for $ \theta_\ast<\frac{\pi}{6} $ and $ \theta_\ast>\frac{\pi}{3} $.  Set $ B_3=B_{\frac{2d_\ast}{3}}(P_\ast)\cap \mn^+ $ and $ B_4=B_{\frac{3d_\ast}{4}}(P_\ast)\cap \mn^+ $. By the Schauder boundary estimate  \cite[Theorem 4.16]{GT98}, one gets
\begin{equation}\label{4-l}
\begin{aligned}
\|\hat \phi_2^+\|_{1,\alpha;B_3}^{(0)}\leq C\bigg(\|\hat \phi_2^+\|_{0,0;B_{4}}+\sum_{i=1}^2\|h_{i+2}\|_{0,\alpha;B_{4}}^{(1)}
+\|\hat h_5^+\|_{1,\alpha;{B_{4}}\cap\{y_2=0\}}
\bigg).
\end{aligned}
\end{equation}
\par Combining  Case 1 and Case 2 yields the following corner estimate:
\begin{equation}\label{4-k}
\|\hat\phi_2^+\|_{1,\alpha;B_{ r_0/2}^+} \leq CM.
\end{equation}
\par {\bf Step 4. The $ C^{1,\alpha} $   estimate of $ \phi_2^+ $ near the right corner points $ (L,0) $ and $ (L,m^+) $.}
\par In this step, we use the method of reflection  to obtain the $ C^{1,\alpha} $ estimate $ \phi^+_2 $ near the right corner points $ (L,0) $ and  $ (L,m^+) $ in $ \mn^+ $.
 Define
\begin{equation}\label{4-25}
\begin{cases}
\e \phi_2^+(y_1,y_2)=\phi_2^+(y_1,y_2),\ \e h_{3}^+(y_1,y_2)= h_{3}^+(y_1,y_2)-h_{3}^+(L,y_2), \\
\e h_{4}^+(y_1,y_2)= h_{4}^+(y_1,y_2),\
\e h_5(y_1)= h_5(y_1), \ \e h_6(y_1)= h_6(y_1),  \quad {\rm{for}} \ 0\leq y_1\leq L;\\
\e \phi_2^+ (y_1,y_2)=\phi_2^+(2L-y_1,y_2),\ \e h_{3}^+(y_1,y_2)=-( h_{3}^+(2L-y_1,y_2)-h_{3}^+(L,y_2)), \\
\e h_{4}^+(y_1,y_2)= h_{4}^+(2L-y_1,y_2),\
\e h_{5}^+(y_1)=  h_{5}^+(2L-y_1),\\
\e h_{6}^+(y_1)=  h_{6}^+(2L-y_1),
 \quad {\rm{for}} \ L\leq y_1\leq 2L.\\
\end{cases}
\end{equation}
Then $ \e \phi_2 $ solves the following problem
\begin{equation}\label{4-26}
\begin{cases}
\begin{aligned}
&\Delta \e \phi_2^+=\sum_{i=1}^2\p_{y_i} \e h_{i+2}^+,\ {\rm{in}}\ E=(0,2L)\times(0,m^+),\\
& \e \phi_2^+(0,y_2)=
\p_{y_1}\e \phi_2^+(2L,y_2)=0,\\
&\e \phi_2^+(y_1,0)=\e h_5(y_1),\\
&\e \phi_2^+(y_1,m^+)=\e h_6(y_1).
\end{aligned}
\end{cases}
\end{equation}
  For a connected domain $  E_l=[\frac{L}{2},\frac{3L}{2}]\times(0,m^+)\subset E $, by the Schauder interior estimate  \cite[Theorem 8.32]{GT98} and  boundary estimate \cite[Corollary 8.36]{GT98}, the following estimate holds:
\begin{equation}\label{4-27}
\begin{aligned}
\| \e \phi_2^+\|_{1,\alpha;\overline{ E_l}}
\leq C\bigg(\|\e \phi_2^+\|_{0,0;E}+
\sum_{i=1}^{2}\|\e h_{i+2}^+\|_{0,\alpha;E}+
\sum_{i=5}^6\|\e h_{i}\|_{1,\alpha;[0,2L]}\bigg)
\leq CM.\\
\end{aligned}
\end{equation}
 \par This, together with   the  estimate \eqref{4-k}, shows  the $ C^{1,\alpha} $ estimate up to the four corners. Away from four corners, we have the standard Schauder interior and boundary estimates. Thus one can conclude that
  \begin{equation}\label{4-28}
  \|  \phi_2^+\|_{1,\alpha;\overline{ \mn^+}}\leq CM.
  \end{equation}
  \par {\bf Step 5. The $ C_{2,\alpha}^{(-1-\alpha)}( \mn^+) $ estimate of $ \phi_2^+ $.}
  \par In this step, we derive the $ C_{2,\alpha}^{(-1-\alpha)}(\mn^+) $ estimate  of $\phi_2^+  $ by a scaling argument. For any fixed point $\mathbf{y}_\ast\in \overline{\mn^+} \setminus(\overline{\Sigma}\cup \overline{\Sigma_w^+}) $, define $ 2d={\rm{dist}}(\mathbf{y}_\ast,\overline{\Sigma}\cup \overline{\Sigma_w^+})$ and a scaled function
  \begin{equation*}
  \phi_2^{(\mathbf{y}_\ast)}(\mathbf{z}):=\frac{1}{d^{1+\alpha}}( \phi_2^+(\mathbf{y}_\ast+d\mathbf{z})-\phi_2^+(\mathbf{y}_\ast)-\n \phi_2^+(\mathbf{y}_\ast) \cdot \mathbf{z}),
  \end{equation*}
  for $ \mathbf{z}\in \{\mathbf{z}\in B_1(0):\mathbf{y}_\ast+d\mathbf{z}\in \mn^+\}=:\mm_1(\mathbf{y}_\ast)$. It follows from \eqref{4-28} that
   \begin{equation}\label{4-29}
  \|  \phi_2^{(\mathbf{y}_\ast)}\|_{0,0;\mm_1(\mathbf{y}_\ast)}\leq \|  \phi_2^+\|_{1,\alpha;\overline{ \mn^+}}\leq CM.
  \end{equation}
\par  Next, substituting $  \phi_2^{(\mathbf{y}_\ast)} $ into the first equation in \eqref{4-16}
  and applying  the standard Schauder  estimate yield that
  \begin{equation}\label{4-30}
  \|  \phi_2^{(\mathbf{y}_\ast)}\|_{2,\alpha;\mm_{1/2}(\mathbf{y}_\ast)}\leq CM.
  \end{equation}
  Therefore, re-scaling back gives
  \begin{equation}\label{4-31}
\|\phi_2^+\|_{2,\alpha;\mn^+}^{(-1-\alpha)}\leq CM.
\end{equation}
Finally, it follows from  the expressions of $ \phi^+_i  $ and $ h_i^+ $ that
\begin{equation}\label{4-35}
\begin{aligned}
\|\phi^+\|_{2,\alpha;\mn^+}^{(-1-\alpha)}
\leq C\bigg(\sum_{i=1}^2\| h_i^+\|_{1,\alpha;\mn^+}^{(-\alpha)}+\|g_{1}^+\|_{2,\alpha;\Sigma_0^+}^{(-1-\alpha)}
+\|g_{cd}\|_{1,\alpha;[0,L]}
+\|g_{2}^+\|_{1,\alpha;[0,L]}
+\|\e\omega_{ex}^+\|_{1,\alpha;\overline{\Sigma_L^+}}\bigg).
\end{aligned}
\end{equation}
 Therefore the proof of Lemma 4.2 is completed.
\end{proof}
\par Now, we come back to the proof of Lemma 4.1. \\
\\
{\bf Proof of Lemma 4.1.}
For the problem \eqref{4-7}, it follows from \eqref{4-8}  that the  coefficients of  \eqref{4-7} satisfy
\begin{equation*}
\|a_{ij}-e_i^+\delta_{ij}\|_{1,\alpha;\mn^+}^{(-\alpha)}\leq \kappa^+,
\end{equation*}
  where $\kappa^+>0$   depends on $ \delta_1 $,  $ \bm U_b^+ $, $ L $ and $ \alpha $.
Thus we  apply the Banach contraction mapping theorem to prove the existence and uniqueness of the solution $  \phi^+\in  C_{2,\alpha}^{(-1-\alpha)}(\mn^+) $ of the problem \eqref{4-7} for sufficiently small $ \kappa^+ $ in the space
\begin{equation*}
\mk=\{v:\|v\|_{2,\alpha;\mn^+}^{(-1-\alpha)}< \infty\}.
\end{equation*}
Define a mapping $ \ma: \mk\rightarrow \mk $. For $ v\in \mk $, we consider the problem \eqref{4-12} with functions $ h_i^+ $ on the right-hand sides defined as
\begin{equation}\label{4-36}
h_i^+=f_i^++\sum_{j=1}^2(e_i^+\delta_{ij}-a_{ij})\p_{y_j}v.
\end{equation}
Then $ h_i^+\in C_{1,\alpha}^{(-\alpha)}(\mn^+ ) $ and
\begin{equation}\label{4-37}
\sum_{i=1}^2\|h_i\|_{1,\alpha;\mn^+}^{(-\alpha)}
\leq C\bigg(\sum_{i=1}^2\|f_i\|_{1,\alpha;\mn^+}^{(-\alpha)}
+\|v\|_{2,\alpha;\mn^+}^{(-1-\alpha)}\bigg).
\end{equation}
By Lemma 4.2, there exists a unique solution $\phi^+ \in C_{2,\alpha}^{(-1-\alpha)}(\mn^+) $ to the problem \eqref{4-12}. We define the mapping $ \ma: \mk\rightarrow \mk $ by setting $ \ma v= \phi^+ $.
\par Now we show that $ \ma $ is a contraction mapping in the norm $ \|\cdot\|_{\mk} $ if $ \kappa^+>0 $ is small. Let $ v_1,v_2\in \mk  $ and $ \phi_k^+=\ma v_k $ for $ k=1,2 $. Then $ \hat \phi^+= \phi_1^+- \phi_2^+ $ satisfies
\begin{equation}\label{4-38}
  \begin{cases}
e_1^+\p_{y_1}^2\hat\phi^+
+e_2^+\p_{y_2}^2\hat\phi^+
=\sum_{i=1}^2\p_{y_i}(\sum_{j=1}^2(e_i^+\delta_{ij}-a_{ij})\p_{y_j}(v_1-v_2)), &\quad {\rm{in}} \quad \mn^+,\\
\hat\phi^+(0,y_2)=0, &\quad {\rm{on}} \quad
  \Sigma_0^+,\\
 \hat\phi^+(y_1,0)=0, &\quad {\rm{on}} \quad
  \Sigma,\\
 \hat\phi^+(y_1,m^+)=0, &\quad {\rm{on}} \quad
  \Sigma_w^+,\\
  \p_{y_1}\hat\phi^+(L,y_2)=0, &\quad {\rm{on}} \quad
  \Sigma_L^+.\\
  \end{cases}
      \end{equation}
      It follows from \eqref{4-13} that one has
     \begin{equation}\label{4-39}
     \| \phi_1^+- \phi_2^+\|_{\mk}\leq \mc^+\kappa^+
      \|v_1^+-v_2^+\|_{\mk}.
      \end{equation}
      Therefore the mapping $ \ma: \mk\rightarrow \mk $ is a contraction mapping in the norm $ \|\cdot\|_{\mk} $ if $ \kappa^+<\frac{1}{\mc^+} $.
      \par For such $ \kappa^+ $, there exists a fixed point $  \phi^+ \in \mk$ satisfying $ \ma  \phi ^+= \phi ^+ $. Then $ \phi^+ $ satisfies \eqref{4-12} with right-hand sides given by \eqref{4-36} computed with $ v=\phi^+ $. Moreover, by \eqref{4-13} and \eqref{4-37},
     there holds
        \begin{equation*}
\begin{aligned}
\|\phi^+\|_{2,\alpha;\mn^+}^{(-1-\alpha)}
\leq \mc^+\bigg(\sum_{i=1}^2\| f_i^+\|_{1,\alpha;\mn^+}^{(-\alpha)}+\|g_{1}^+\|_{2,\alpha;\Sigma_0^+}^{(-1-\alpha)}
+\|g_{cd}\|_{1,\alpha;[0,L]}
+\|g_{2}^+\|_{1,\alpha;[0,L]}
+\|\e\omega_{ex}^+\|_{1,\alpha;\overline{\Sigma_L^+}}\bigg).\\
\end{aligned}
\end{equation*}
Then Lemma 4.1 is proved.
  \subsection{Solving the nonlinear  boundary value problem by the fixed point argument}\noindent
  \par For a given function $   \bar\phi^+\in \mj_{upp}(\delta_1) $,
  it follows from Lemma  4.1 that the problem \eqref{4-7} has a unique solution $\phi^{+}\in C_{2,\alpha}^{(-1-\alpha)}(\mn^+) $ satisfying the estimate in \eqref{4-11}.  This, together with \eqref{4-6}, yields
   \begin{equation}\label{4-40}
\begin{aligned}
\|\phi^+\|_{2,\alpha;\mn^+}^{(-1-\alpha)}
&\leq \mc^+\bigg(\sum_{i=1}^2\| f_i^+\|_{1,\alpha;\mn^+}^{(-\alpha)}+\|g_{1}^+\|_{2,\alpha;\Sigma_0^+}^{(-1-\alpha)}
+\|g_{cd}\|_{1,\alpha;
[0,L]}
+\|g_{2}^+\|_{1,\alpha;
[0,L]}
+\|\e\omega_{ex}^+\|_{1,\alpha;\overline{\Sigma_L^+}}\bigg)\\
& \leq
  \mc^+
  \bigg(\delta_1 \sigma_v+\sigma_v+\|\eta\|_{0,\alpha;[0,L]}\bigg),
  \end{aligned}
  \end{equation}
  where   $ \mc^+>0 $
  depends only  on  $ {\bm U}_b^+ $, $ L $ and $\alpha $.
 \par  Similarly,   define
 \begin{equation}\label{4-41}
 \mj_{low}(\delta_2)
 =\{\phi^-:\phi^-(0,0)=0, \|\phi^-\|_{2,\alpha; \mn^-}^{(-1-\alpha)}
 \leq
 \delta_2\},
 \end{equation}
 where $ \delta_2 $ is a positive constant to be determined later.
 Then for given $ \bar\varphi^- $ such that $ \bar\phi^-=\bar\varphi^--\varphi^-_b\in \mj_{low}(\delta_2) $, we  solve the following problem:
 \begin{equation}\label{4-42}
 \begin{cases}
   \sum_{i,j=1,2}\p_{y_i}(a_{ij}(\n
   \bar\phi^-)\p_{y_j}\phi^-)=\sum_{i=1}^2\p_{y_i}f_i^- &\quad {\rm{in }} \quad      \mn^-,\\
     \bar\phi^-(0,y_2)=g_{1}^-(y_2), &\quad {\rm{on}} \quad
   \Sigma_0^-,\\
  \phi^-(y_1,0)=g_{cd}(y_1), &\quad {\rm{on}} \quad
   \Sigma,\\
   \phi^-(y_1,-m^-)=g_{2}^-(y_1), &\quad {\rm{on}} \quad
   \Sigma_w^-,\\
   \p_{y_1}\phi^-(L,y_2)=\e \omega_{ex}^-(y_2), &\quad {\rm{on}} \quad
   \Sigma_L^-,\\
   \end{cases}
 \end{equation}
  where
   \begin{equation*}
      \begin{aligned}
      &f_i^-=W_i(\n\varphi_b^-, A_{b}^-,  B_{b}^-)-W_i(\n\varphi_b^-, \e
      A_{en}^-, \e B_{en}^-) , \quad i=1,2,\\
       &g_{1}^-(y_2)=-\int_{y_2}^0\left(\frac{1}{\e J_{en}^-(s)}-\frac{1}{ J_b^-}\right)\de s,\\
    &g_{2}^-(y_1)=g^-(y_1)+1-\int_{-m^-}^0\left(\frac{1}{\e J_{en}^-(s)}-\frac{1}{ J_b^-}\right)\de s,\\
    &\e\omega_{ex}^-(y_2)=\omega_{ex}(\varphi_b^-(y_2)+\bar\phi^-(L,y_2)).
    \end{aligned}
    \end{equation*}
     Then similar arguments as in Lemma 4.1 yield that \eqref{4-42}
 has a unique solution $  \phi^-\in  C_{2,\alpha}^{(-1-\alpha)}(\mn^-) $ satisfying
 \begin{equation}\label{4-43}
 \begin{aligned}
  \|\phi^-\|_{2,\alpha;\mn^-}
  ^{(-1-\alpha)}
  &\leq \mc^-\bigg(\sum_{i=1}^2\|f_i^-\|_{1,\alpha;\mn ^-}
  ^{(-\alpha)}+\|g_{1}^-\|_{2,\alpha;\Sigma _{0}^-}^{(-1-\alpha)}
  +\|g_{cd}\|_{1,\alpha;[0,L]} +\|g_{2}^-\|_{1,\alpha;[0,L]}
  +\|\e\omega_{ex}^-\|_{1,\alpha;\overline{\Sigma_L^-}}\bigg)\\
  &\leq
 \mc^-\left(\delta_2\sigma_v +\sigma_v+\|\eta\|_{0,\alpha;[0,L]}\right),
  \end{aligned}
  \end{equation}
  where   $ \mc^->0 $
  depends only  on  $ {\bm U}_b^- $, $ L $ and $\alpha $.
\par Define an iteration set
\begin{equation*}
\mj(\delta_1,\delta_2)=\mj_{upp}(\delta_1)\times \mj_{low}(\delta_2),
\end{equation*}
and a map $ \mt $ as follows
\begin{equation}\label{4-44}
\mt( \bar\phi^+, \bar\phi^-)=( \phi^+, \phi^-),
\quad{\rm{ for \ each}} \ ( \bar\phi^+, \bar\phi^-)\in \mj(\delta_1,\delta_2).
\end{equation}
Then it follows from \eqref{4-40} and \eqref{4-43} that
$ ( \phi^+, \phi^-) $ satisfies
\begin{equation}\label{4-45}
\begin{aligned}
\|\phi^+\|_{2,\alpha;\mn^+}^{(-1-\alpha)}+
\| \phi^-\|_{2,\alpha;\mn^-}^{(-1-\alpha)}
\leq \mc_1
\left((\delta_1+\delta_2)\sigma_v+\sigma_v+
\|\eta\|_{0,\alpha;[0,L]}\right),
\end{aligned}
\end{equation}
where  $ \mc_1>0 $  depends  on $ (\bm U_b^+,\bm U_b^-) $, $ L $
and $ \alpha $.   We assume that
\begin{equation}\label{4-46}
\|\eta\|_{0,\alpha;[0,L]}\leq \delta_3,
\end{equation}
where $ \mc_1 \delta_3\leq \frac{\delta_1+\delta_2}{2}  $
with $ \mc_1 $ given in \eqref{4-45}.  Let $\sigma_1=\frac{1}{4(1+\mc_1)}$ and choose $\delta_1=\delta_2= 2\mc_1\sigma_v$. Then if $ \sigma_v\leq \sigma_1 $, one has
\begin{equation*}
\begin{aligned}
\| \phi^+\|_{2,\alpha;\mn^+}^{(-1-\alpha)}+
\| \phi^-\|_{2,\alpha;\mn^-}^{(-1-\alpha)} \leq  \frac{1}{2}(\delta_1+\delta_2)
+\mc_1\delta_3 \leq \delta_1+\delta_2. \end{aligned}
\end{equation*}
Hence $ \mt $ maps $ \mj(\delta_1,\delta_2) $ into itself.
  \par Next, we will show that $ \mt $ is a
  contraction mapping in $ \mj(\delta_1,\delta_2) $. Let $(\bar\phi_i^+, \bar\phi_i^-)\in \mj(\delta_1,\delta_2)$,  $ i=1,2 $, we  have $ ( \phi_i^+, \phi_i^-)=\mt(\bar\phi_i^+, \bar\phi_i^-) $.
Define
\begin{equation*}
\bar\Phi^\pm= \bar\phi_1^\pm- \bar\phi_2^\pm,\
\Phi^\pm= \phi_1^\pm- \phi_2^\pm.
\end{equation*}
Then  it follows from \eqref{4-40} and  \eqref{4-43} that one obtains
\begin{equation}\label{4-48}
\begin{aligned}
\sum_{I=\pm}\| \Phi^I\|_{2,\alpha;\mn^I}^{(-1-\alpha)}
&\leq \mc_2\sum_{I=\pm}\bigg((\| a_{ij}(\n
   \bar\phi_1^I)-a_{ij}(\n
   \bar\phi_2^I))\p_{y_j}\phi_2^I\|_{1,\alpha;\mn^I}^{(-\alpha)}\\
   &\qquad\qquad+\|  \omega_{ex}(\varphi_b^I(y_2)+\bar\phi_1^I(L,y_2))
   -\omega_{ex}(\varphi_b^I(y_2)+\bar\phi_2^I(L,y_2))\|
   _{1,\alpha;\overline{\Sigma_L^I}}\bigg)\\
&\leq\mc_2 (\delta_1+\delta_2+\sigma_v)\sum_{I=\pm}\| \bar\Phi^I\|_{2,\alpha;\mn^I}^{(-1-\alpha)}.
\end{aligned}
\end{equation}
Setting
\begin{equation}\label{4-49}
 \sigma_2=\min\left\{\sigma_1,\frac{1}{
4\mc_2(4\mc_1+1)} \right\}.
\end{equation}
Then for $ \sigma_v\leq \sigma_2 $,
$ \mc_2 (\delta_1+\delta_2+\sigma_v)=\mc_2 (4\mc_1+1)\sigma_v\leq \frac{1}{4}  $, hence the mapping $ \mt $ is a contraction mapping so that $\mt $ has a unique fixed point in $ \mj(\delta_1,\delta_2)$.
  \section{The construction of the contact discontinuity curve}\noindent
  \par  To complete the proof of Theorem 3.2, we  use Theorem 3.3 to find the contact discontinuity  $ g_{cd}(y_1) $ such that \eqref{3-23} is satisfied. Firstly,
   define a Banach space
\begin{equation*}
 V=\{\eta: \|\eta \|_{0,\alpha;[0,L]}
  < \infty\}.
  \end{equation*}
 Set
  \begin{equation}\label{5-1}
 V_{\delta_3}=\{\eta\in V:  \|\eta \|_{0,\alpha;[0,L]}
 \leq \delta_3\},
  \end{equation}
  where  $ \delta_3$ is defined in \eqref{4-46}. Then for any $ \eta\in  V_{\delta_3}$,
   the problem \eqref{3-19}-\eqref{3-22} has a unique solution $ (\varphi^+,  \varphi^-) $ satisfying
 \begin{equation}\label{5-2}
   \|\varphi^+-\varphi_b^+\|_{2,\alpha;\mn^+}^{(-1-\alpha)}+
   \|\varphi^--\varphi_b^-\|_{2,\alpha;\mn^-}
   ^{(-1-\alpha)}\leq 4\mc_1\sigma_v.
   \end{equation}
  \par Let
\begin{equation*}
   \begin{aligned}
  V_0&= C^{1,\alpha}([0,1])\times C^{1,\alpha}([0,1])\times C^{1,\alpha}([0,1])\times
    C^{1,\alpha}([0,L])
   \times C^{1,\alpha}([-1,0])\\
  &\quad\quad\times
   C^{1,\alpha}([-1,0])\times C^{1,\alpha}([-1,0])
  \times
     C^{1,\alpha}([0,L])\times
  C^{2,\alpha}([g^-(L),g^+(L)]).
   \end{aligned}
    \end{equation*}
    Then we set
    \begin{equation}\label{5-3}
    V_{\delta_4}=\{
    \bm\zeta_0\in V_0:\|\bm \zeta_0-\bm \zeta_b\|_{ V_0}\leq \delta_4\},
    \end{equation}
    where \begin{equation*}
     \bm \zeta_0=( A_{en}^+, B_{en}^+, J_{en}^+, g^+,
   A_{en}^-,B_{en}^-,J_{en}^-, g^-, \omega_{ex}), \ {\rm{and}} \
     \bm \zeta_b= (A_{b}^+, B_b^+,J_b^+, 1, A_b^-, B_b^-,J_b^-, -1,0).
   \end{equation*}
     \par Define a map $ \mq:   V_{\delta_4} \times  V_{\delta_3}\rightarrow  V $ by
  \begin{equation}\label{5-4}
  \mq(\bm \zeta_0,\eta):=\left(W_2(\n \varphi_b^++\n \phi^+ ,\e A_{en}^+,\e B_{en}^+ )-  W_2(\n \varphi_b^-+\n \phi^-,\e A_{en}^-,\e B_{en}^- )\right)(y_1,0).
  \end{equation}
   Hence \eqref{3-23} can be written as the equation
  \begin{equation}\label{5-5}
  \mq(\bm \zeta_0,\eta)=0.
  \end{equation}
     It suffices to verify the conditions $ \rm(i) $, $ \rm(ii) $ and $ \rm(iii) $ in Theorem 3.3.
  \par  Obviously,
  \begin{equation*}
  \mq(\bm \zeta_b,0)=0.
  \end{equation*}
  Next, we will divided into two steps to  verify $ \rm(ii) $ and $ \rm(iii) $.
\par {\bf Step 1. Differentiability of $ \mq $.}
  \par
   Given any $ \eta,\eta_1\in V_{\delta_3} $, and  $ \tau>0 $,
 let $  \phi^\pm_{\e \eta} $ and
  $  \phi^\pm_{\eta} $ be the solutions of the following equations
  \begin{equation}\label{5-6}
  \begin{cases}
  \begin{aligned}
  &\sum_{i,j=1,2}\p_{y_i}(a_{ij}(\phi^\pm_{\e \eta})\p_{y_j}\phi^\pm_{\e \eta})=\sum_{i=1}^2\p_{y_i}f_i^\pm, \\
&\phi^\pm_{\e \eta}(0,y_2)=g_{1}^\pm(y_2), \\
 &\phi^\pm_{\e \eta}(y_1,0)=\int_{0}^{y_1}(\eta+\tau \eta_1)(s)\de s, \\
  &\phi^\pm_{\e \eta}(y_1,\pm m^\pm)=g_{2}^\pm(y_1),\\
  &\p_{y_1}\phi^\pm_{\e \eta}(L,y_2)=\omega_{ex}(\varphi_b^\pm(y_2)+\phi^\pm_{\e \eta}(L,y_2)),\\
  \end{aligned}
  \end{cases}
\end{equation}
and
\begin{equation}\label{5-7}
  \begin{cases}
  \begin{aligned}
&\sum_{i,j=1,2}\p_{y_i}(a_{ij}(\phi^\pm_{ \eta})\p_{y_j}\phi^\pm_{ \eta})=\sum_{i=1}^2\p_{y_i}f_i^\pm, \\
&\phi^\pm_{\eta}(0,y_2)=g_{1}^\pm(y_2), \\
 &\phi^\pm_{ \eta}(y_1,0)=\int_{0}^{y_1}\eta(s)\de s, \\
 & \phi^\pm_{ \eta}(y_1,\pm m^\pm)=g_{2}^\pm(y_1),,\\
  &\p_{y_1}\phi^\pm_{ \eta}(L,y_2)=\omega_{ex}(\varphi_b^\pm(y_2)+\phi^\pm_{ \eta}(L,y_2)).\\
\end{aligned}
  \end{cases}
\end{equation}
\par Denote
  \begin{equation*}
  \phi^{\pm}_{\tau}=\frac{ \phi^\pm_{\e \eta}- \phi^\pm_{\eta}}{\tau}.
  \end{equation*}
    Then it follows from \eqref{5-6} and \eqref{5-7} that $  \phi^{\pm}_{\tau}$ satisfy the following equations:
  \begin{equation}\label{5-8}
  \begin{cases}
  \begin{aligned}
  &\sum_{i,j=1,2}\p_{y_i}(a_{ij}(\n
   \phi_{\eta}^\pm)\p_{y_j}\phi^{\pm}_{\tau})
  =-\sum_{i,j=1,2}\p_{y_i}\left(\frac{a_{ij}(\n
   \phi_{\e \eta}^\pm)-a_{ij}(\n
   \phi_\eta^\pm)}{\tau}\p_{y_j}\phi_{\e \eta}^{\pm}\right), \\
   &  \phi^{\pm}_{\tau}(0,y_2)=0,\\
  &\phi^{\pm}_{\tau}(y_1,0)= \int_{0}^{y_1}\eta_1(s)\de s, \\
  &\phi^{\pm}_{\tau}(y_1,\pm m^\pm)=0, \\
   &\p_{y_1}\phi^{\pm}_{\tau}(L,y_2)= \frac{\omega_{ex}(\varphi_b^\pm(y_2)+\phi^\pm_{\e \eta}(L,y_2))-\omega_{ex}(\varphi_b^\pm(y_2)+\phi^\pm_{ \eta}(L,y_2))}{\tau}.\\
    \end{aligned}
    \end{cases}
  \end{equation}
 Similar to the proof of Lemma 4.1,   the following estimate holds:
  \begin{equation}\label{5-9}
  \begin{aligned}
  \sum_{I=\pm}\|\phi^{I}_{\tau}\|_{2,\alpha; \mn^I}
  ^{(-1-\alpha)}
  &\leq \mc_3\sum_{I=\pm}\left(\left(\| \phi^I_{\e \eta}\|_{2,\alpha;\mn^I}
  ^{(-1-\alpha)}
   +\| \omega_{ex}\|_{2,\alpha;[g^-(L),g^+(L)]}\right)
   \|\phi^{I}_{\tau}\|
  _{2,\alpha;\mn^I}
  ^{(-1-\alpha)}+\|  \eta_1\|_{0,\alpha ;[0,L]}\right)\\
 &\leq \mc_3\left(\delta_1+\delta_2+\delta_3
  +\tau\|  \eta_1\|_{0,\alpha ;[0,L]}+\sigma_v\right)
  \sum_{I=\pm}\|\phi^{I}_{\tau}\|
  _{2,\alpha;\mn^I}^{(-1-\alpha)}
  +\mc_3\| \eta_1\|_{0,\alpha ;[0,L]}\\
  &\leq \mc_3\left((3+4 \mc_1)\sigma_v
  +\tau\| \eta_1\|_{0,\alpha ;[0,L]}\right)
  \sum_{I=\pm}\|\phi^{I}_{\tau}\|_{2,\alpha;\mn^I}
  ^{(-1-\alpha)}
  +\mc_3\|  \eta_1\|_{0,\alpha ;[0,L]},\\
  \end{aligned}
  \end{equation}
  where   $ \mc_3>0 $  depends  on $ ({\bm U}_b^+,{\bm U}_b^-) $, $ L$ and $ \alpha $.
   Choosing $ \tau_1>0 $ such that  $ \mc_3\tau_1\|  \eta_1\|_{0,\alpha ;[0,L]}\leq \frac{1}{4} $
   and setting
  \begin{equation}\label{5-10}
  \sigma_3=\min\left\{\sigma_2,\frac{1}{4\mc_3(3+4\mc_1)}\right\},
   \end{equation}
where $ \sigma_2 $ is defined in \eqref{4-49}. Thus for $ \tau \in (0,\tau_1) $ and $ \sigma_v\leq \sigma_3 $, one has
  \begin{equation}\label{5-11}
  \sum_{I=\pm}\|\phi^{I}_{\tau}\|_{2,\alpha; \mn^I}
  ^{(-1-\alpha)}\leq
  2\mc_3\|  \eta_1\|_{0,\alpha ;[0,L]}.
  \end{equation}
  Therefore, there exists a subsequence $\{\tau_k\}_{k=1}^\infty $ such that $ \phi^{\pm}_{\tau_k} $ converge to $ \phi^{\pm}_0$  in  $ C_{2,\alpha^\prime}^{(-1-\alpha^\prime)}(\mn^\pm) $
   as $ \tau_k\rightarrow 0 $ for some $ 0<\alpha^\prime<\alpha $.
The estimate \eqref{5-11} also implies that  $ \phi^{\pm}_0\in C_{2,\alpha}^{(-1-\alpha)}(\mn^\pm) $ and
  \begin{equation}\label{5-13}
  \sum_{I=\pm}\|\phi^{I}_{0}\|
  _{2,\alpha;\mn^I}^{(-1-\alpha)}
  \leq\mc_3\|  \eta_1\|_{0,\alpha ;[0,L]}.
  \end{equation}
  \par Define a map $ D_\eta\mq(\bm\zeta_0,\eta) $ by
  \begin{equation}\label{5-14}
  \begin{aligned}
 D_\eta\mq(\bm\zeta_0,\eta)(\eta_1)
  &=\sum_{i=1}^2\bigg(\p_{\p_{y_i}\phi_{\eta}^{+}} W_2(\n \varphi_b^{+}+\n \phi^{+}_{\eta},\e A_{en}^+,\e B_{en}^+)\p_{y_i}\phi^{+}_{0}\\
  &\qquad\quad-\p_{\p_{y_i}\phi^-_{\eta}}W_2(\n \varphi_b^{-}+\n \phi^{-}_{\eta},\e A_{en}^-,\e B_{en}^-)\p_{y_i}\phi^{-}_{0}\bigg)(y_1,0),\\
  \end{aligned}
  \end{equation}
   which is  a linear map from $   V $ to $ V $.
   \par  Next, we need to show that $ D_\eta\mq(\bm\zeta_0,\eta) $  is  the Fr\'{e}chet derivative of the functional $ \mq(\bm\zeta_0,\eta)$ with respect to $ \eta $.
To achieve this, we first consider the estimate of
  $ \phi^{\pm}_{\tau}-\phi^{\pm}_{0} $.
   It follows from \eqref{5-8} that
  $\phi^{\pm}_{0} $ satisfy the following equations:
  \begin{equation}\label{5-12}
  \begin{cases}
  \begin{aligned}
   &\sum_{i,j=1,2}\p_{y_i}(a_{ij}(\n
   \phi_{\eta}^\pm)\p_{y_j}\phi^{\pm}_{0})
  =-\sum_{i,j=1,2}\p_{y_i}\left(\sum_{k=1}^2\p_{\p_{y_k}
   \phi_\eta^\pm}a_{ij}(\n
   \phi_\eta^\pm)\p_{y_k}
   \phi_0^\pm \p_{y_j}\phi_{ \eta}^{\pm}\right), \\
   &  \phi^{\pm}_{0}(0,y_2)=0,\\
  &\phi^{\pm}_{0}(y_1,0)=\int_{0}^{y_1}\eta_1(s)\de s, \\
  &\phi^{\pm}_{0}(y_1,\pm m^\pm)=0, \\
   &\p_{y_1}\phi^{\pm}_{0}(L,y_2)= \omega_{ex}^\prime(\varphi_b^\pm(y_2)+\phi^\pm_{ \eta}(L,y_2))\phi^{\pm}_{0}(L,y_2).\\
    \end{aligned}
    \end{cases}
  \end{equation}
    Thus it follows from \eqref{5-12} and \eqref{5-8} that $ \phi^{\pm}_{\tau}-\phi^{\pm}_{0}$
satisfy
   \begin{equation}\label{5-15}
  \begin{cases}
  \begin{aligned}
  &\sum_{i,j=1,2}\p_{y_i}(a_{ij}(\n
   \phi_{\eta}^\pm)\p_{y_j}(\phi^{\pm}_{\tau}-\phi^{\pm}_{0})\\
  &\quad=-\sum_{i,j=1,2}\p_{y_i}\left(\frac{a_{ij}(\n
   \phi_{\e \eta}^\pm)-a_{ij}(\n
   \phi_\eta^\pm)}{\tau}\p_{y_j}\phi_{\e \eta}^{\pm}-\sum_{k=1}^2\p_{\p_{y_k}
   \phi_\eta^\pm}a_{ij}(\n
   \phi_\eta^\pm)\p_{y_k}
   \phi_0^\pm\p_{y_j}\phi_{ \eta}^{\pm}\right), \\
   &  (\phi^{\pm}_{\tau}-\phi^{\pm}_{0})(0,y_2)=0,\\
  &(\phi^{\pm}_{\tau}-\phi^{\pm}_{0})(y_1,0)=0, \\
  &(\phi^{\pm}_{\tau}-\phi^{\pm}_{0})(y_1,\pm m^\pm)=0, \\
   &\p_{y_1}(\phi^{\pm}_{\tau}-\phi^{\pm}_{0})(L,y_2)=
   \frac{\omega_{ex}(\varphi_b^\pm(y_2)+\phi^\pm_{\e \eta}(L,y_2))-\omega_{ex}(\varphi_b^\pm(y_2)+\phi^\pm_{ \eta}(L,y_2))}{\tau}\\
    &\qquad\qquad\qquad\qquad\qquad- \omega_{ex}^\prime(\varphi_b^\pm(y_2)+\phi^\pm_{ \eta}(L,y_2))\phi^{\pm}_{0}(L,y_2).\\
    \end{aligned}
    \end{cases}
  \end{equation}
  By the direct computation, one gets
\begin{equation*}
\begin{aligned}
&\sum_{I=\pm}\left\|\frac{a_{ij}(\n
   \phi_{\e \eta}^I)-a_{ij}(\n
   \phi_\eta^I)}{\tau}\p_{y_j}\phi_{\e \eta}^{I}-\sum_{k=1}^2\p_{\p_{y_k}
   \phi_\eta^I}a_{ij}(\n
   \phi_\eta^I)\p_{y_k}
   \phi_0^I\p_{y_j}\phi_{ \eta}^{I}\right\|_{1,\alpha;\mn^I}^{(-\alpha)}\\
&\leq\sum_{I=\pm}\left\|\left(\frac{a_{ij}(\n
   \phi_{\e \eta}^I)-a_{ij}(\n
   \phi_\eta^I)}{\tau}-\sum_{k=1}^2\p_{\p_{y_k}
   \phi_\eta^I}a_{ij}(\n
   \phi_\eta^I)\p_{y_k}
   \phi_0^I\right)\p_{y_j}\phi_{\e \eta}^{I}\right\|_{1,\alpha;\mn^I}^{(-\alpha)}\\
   &\quad+ \sum_{I=\pm}\left\|\sum_{k=1}^2\p_{\p_{y_k}
   \phi_\eta^I}a_{ij}(\n
   \phi_\eta^I)(\p_{y_k}
   \phi_{\e \eta}^I-\p_{y_k}
   \phi_{\eta}^I)\p_{y_j}
   \phi_0^I\right\|_{1,\alpha;\mn^I}^{(-\alpha)}\\
&\leq
\mc\|\phi^{I}_{\e \eta}\|_{2,\alpha;\mn^I}^{(-1-\alpha)}
\|\phi^{I}_{\tau}-\phi^{I}_{0}\|_{2;\alpha,\mn^I}
^{(-1-\alpha)}+\mc\sum_{I=\pm}\|\phi^{I}_{0}\|
_{2,\alpha;\mn^I}^{(-1-\alpha)}
\|\phi^{I}_{\tau}\|_{2,\alpha;\mn^I}^{(-1-\alpha)}\tau,\\
\end{aligned}
  \end{equation*}
  and
  \begin{equation*}
\begin{aligned}
&\sum_{I=\pm}\left\| \frac{\omega_{ex}(\varphi_b^I(y_2)+\phi^I_{\e \eta}(L,y_2))-\omega_{ex}(\varphi_b^I(y_2)+\phi^I_{ \eta}(L,y_2))}{\tau}\right.\\
&\qquad\quad\left.- \omega_{ex}^\prime(\varphi_b^I(y_2)+\phi^I_{ \eta}(L,y_2))\phi^{I}_{0}(L,y_2)\right\|
_{1,\alpha;\overline{\Sigma_L^I}}\\
&\leq \sum_{I=\pm}\left\|\int_{0}^1\omega_{ex}^\prime(\varphi_b^I(y_2)+\phi^I_{ \eta}(L,y_2)+t\tau\phi^{I}_{\tau}(L,y_2))\de t(\phi^{I}_{\tau}-\phi^{I}_{0})(L,y_2)\right\|
_{1,\alpha;\overline{\Sigma_L^I}}\\
&\quad+ \sum_{I=\pm}\left\|\int_{0}^1\left(\omega_{ex}^\prime(\varphi_b^I(y_2)+\phi^I_{ \eta}(L,y_2)+t\tau\phi^{I}_{\tau}(L,y_2))\right.\right.\\
&\qquad\qquad\quad\left.\left.-\omega_{ex}^\prime(\varphi_b^I(y_2)+\phi^I_{ \eta}(L,y_2))\right)\de t\phi^{I}_{0}(L,y_2) \right\|
_{1,\alpha;\overline{\Sigma_L^I}}\\
&\leq \mc\| \omega_{ex}\|_{2,\alpha;[g^-(L),g^+(L)]}\sum_{I=\pm}
\left(\|\phi^{I}_{\tau}-\phi^{I}_{0}\|_{2,\alpha;\mn^I}^{(-1-\alpha)}
+\|\phi^{I}_{0}\|
_{2,\alpha;\mn^I}^{(-1-\alpha)}
\|\phi^{I}_{\tau}\|_{2,\alpha;\mn^I}^{(-1-\alpha)}\tau\right).
\end{aligned}
  \end{equation*}
 Hence the following estimate can be derived:
  \begin{equation}\label{5-16}
  \begin{aligned}
  &\sum_{I=\pm}\|\phi^{I}_{\tau}-\phi^{I}_{0}\|
  _{2,\alpha;\mn^I}^{(-1-\alpha)}\\
  &\leq \mc\sum_{I=\pm}\|\phi^{I}_{0}\|
_{2,\alpha;\mn^I}^{(-1-\alpha)}
\|\phi^{I}_{\tau}\|_{2,\alpha;\mn^I}^{(-1-\alpha)}\tau
+\mc\|\phi^{I}_{\e \eta}\|_{2,\alpha;\mn^I}^{(-1-\alpha)}
\|\phi^{I}_{\tau}-\phi^{I}_{0}\|_{2,\alpha;\mn^I}
^{(-1-\alpha)}\\
&\quad+\mc\| \omega_{ex}\|_{2,\alpha;[g^-(L),g^+(L)]}\sum_{I=\pm}
\left(\|\phi^{I}_{\tau}-\phi^{I}_{0}\|_{2,\alpha;\mn^I}^{(-1-\alpha)}
+\|\phi^{I}_{0}\|
_{2,\alpha;\mn^I}^{(-1-\alpha)}
\|\phi^{I}_{\tau}\|_{2,\alpha;\mn^I}^{(-1-\alpha)}\tau\right)\\
 &\leq  \mc_4\left(\delta_1+\delta_2+\delta_3
  +\tau\| \eta_1\|_{0,\alpha ;[0,L]}+\sigma_v\right)
  \sum_{I=\pm}\|\phi^{I}_{\tau}-\phi^{I}_{0}\|
  _{2,\alpha;\mn^I}^{(-1-\alpha)}
  +\mc_4
  \left(\| \eta_1\|_{0,\alpha ;[0,L]}
  \right)^2\tau\\
  &\leq \mc_4\left((3+4 \mc_1)\sigma_v
  +\tau\|  \eta_1\|_{0,\alpha ;[0,L]}\right)
  \sum_{I=\pm}\|\phi^{I}_{\tau}-\phi^{I}_{0}\|
  _{2,\alpha;\mn^I}^{(-1-\alpha)}
  +\mc_4
  \left(\| \eta_1\|_{0,\alpha ;[0,L]}
  \right)^2\tau,\\
  \end{aligned}
  \end{equation}
  where   $ \mc_4>0 $  depends  on $ ({\bm U}_b^+,{\bm U}_b^-) $, $ L$ and $ \alpha $. Choosing  $ \tau_2>0 $  such that  $ \mc_4\tau_2\|  \eta_1\|_{0,\alpha ;[0,L]}\leq \frac{1}{8} $ and setting
  \begin{equation}\label{5-17}
   \sigma_4=\min\left\{\sigma_3,\frac{1}{8\mc_4(3+4\mc_1)}\right\}.
   \end{equation}
  Then for $ \tau \in (0,\tau_2) $ and $ \sigma_v\leq \sigma_4 $, one has
\begin{equation}\label{5-18}
  \sum_{I=\pm}\|\phi^{I}_{\tau}-\phi^{I}_{0}\|
  _{2,\alpha;\mn^I}^{(-1-\alpha)}\leq 4\mc_4
  \left(\|  \eta_1\|_{0,\alpha ;[0,L]}
  \right)^2\tau.
  \end{equation}
 \par By the definition of $ \mq(\bm\zeta_0,\eta) $ and $ D_\eta\mq(\bm\zeta_0,\eta) $, we have
  \begin{equation}\label{5-19}
  \begin{aligned}
  &\left\| \mq(\bm\zeta_0,\eta+\tau \eta_1)-\mq(\bm\zeta_0,\eta)- D_\eta\mq(\bm\zeta_0,\eta)( \tau\eta_1)\|_{0,\alpha ;[0,L]}\right.\\
  &=\|(W_2(\n \varphi_b^I+\n \phi^I_{\e \eta},\e A_{en}^I,\e B_{en}^I)
  - W_2(\n\varphi_b^I+\n \phi^I_{\eta},\e A_{en}^I,\e B_{en}^I)\\
  &\quad\left.-\tau\sum_{i=1}^2\left(\p_{\p_{y_i}\phi_{\eta}^{+}} W_2(\n \varphi_b^{+}+\n \phi^{+}_{\eta},\e A_{en}^+,\e B_{en}^+)\p_{y_i}\phi^{+}_{0}\right.\right.\\
  &\qquad\quad\quad\left.\left.-\p_{\p_{y_i}\phi^-_{\eta}}W_2(\n \varphi_b^{-}+\n \phi^{-}_{\eta},\e A_{en}^-,\e B_{en}^-)\p_{y_i}\phi^{-}_{0}\right)
  \right\|_{0,\alpha ;[0,L]}\\
  &\leq \sum_{I=\pm}\left\|  W_2(\n \varphi_b^I+\n \phi^I_{\e \eta},\e A_{en}^I,\e B_{en}^I)
  - W_2(\n\varphi_b^I+\n \phi^I_{\eta},\e A_{en}^I,\e B_{en}^I)\right.\\
  &\quad\quad\quad\left.-\tau\sum_{i=1}^2\p_{\p_{y_i} \phi^I_{\eta}} W_2(\n\varphi_b^I+\n \phi^I_{\eta},\e A_{en}^I,\e B_{en}^I)\p_{y_i}\phi^{I}_{\tau}\right\|
  _{0,\alpha ;[0,L]}\\
  &\quad+\tau \sum_{I=\pm}\left\|\sum_{i=1}^2\p_{\p_{y_i} \phi^I_{\eta}} W_2(\n\varphi_b^I+\n \phi^I_{\eta},\e A_{en}^I,\e B_{en}^I)\p_{y_i}(
  \phi^{I}_{\tau}-\phi^{I}_{0})
  \right\|
  _{0,\alpha ;[0,L]}\\
   &\leq \mc\tau^2\sum_{I=\pm}\left(\|\phi^{I}_{\tau}\|
_{2,\alpha;\mn^I}^{(-1-\alpha)}
  \right)^2+ \mc\tau\sum_{I=\pm}\|\phi^{I}_{\tau}
  -\phi^{I}_{0}\|
 _{2,\alpha;\mn^I}^{(-1-\alpha)}\\
&\leq \mc\tau^2
  \left(\| \eta_1\|_{0,\alpha ;[0,L]}
  \right)^2,\\
  \end{aligned}
  \end{equation}
   where   $ \mc>0 $  depends  on $ ({\bm U}_b^+,{\bm U}_b^-) $, $ L $ and $ \alpha $. This implies that
  \begin{equation*}
  \frac{\| \mq(\bm\zeta_0,\eta+\tau \eta_1)-\mq(\bm\zeta_0,\eta)- D_\eta\mq(\bm\zeta_0,\eta)( \tau\eta_1)\|_{0,\alpha ;[0,L]}}
  {\tau\|\eta_1\|_{0,\alpha ;[0,L]}}\rightarrow 0
  \end{equation*}
  as $ \tau\rightarrow 0 $. Thus $ D_\eta\mq(\bm\zeta_0,\eta) $ is   the Fr\'{e}chet derivative of the functional $ \mq(\bm\zeta_0,\eta)$ with respect to $ \eta $.
\par It remains to prove  the continuity of the map $ \mq(\bm\zeta_0,\eta) $ and $ D_\eta\mq(\bm\zeta_0,\eta) $. For any fixed $(\bm \zeta_0, \eta)\in V_0\times V $, we assume that $ (\bm \zeta_0^k,\eta^k)\rightarrow (\bm \zeta_0, \eta) $ in $ V_0\times V $ as $ k\rightarrow \infty $. Then we first show that as $ k\rightarrow \infty $,
  \begin{equation}\label{5-20}
  \mq(\bm\zeta_0^k,\eta^k)\rightarrow  \mq(\bm\zeta_0,\eta), \quad {\rm{in}}\quad  V.
  \end{equation}
  By \eqref{5-7}, the solutions $ \phi^\pm_{ \eta^k}$ corresponding to $ (\bm\zeta_0^k,\eta^k) $ satisfy
  \begin{equation}\label{5-21}
  \begin{cases}
  \begin{aligned}
&\sum_{i,j=1,2}\p_{y_i}(a_{ij}(\n
   \phi_{\eta^k}^\pm,\e A_{en}^{\pm,k},\e B_{en}^{\pm,k})\p_{y_j}\phi^\pm_{ \eta^k})=\sum_{i=1}^2\p_{y_i}f_i^{\pm,k}, \\
&\phi^\pm_{\eta^k}(0,y_2)=g_{1}^{\pm,k}(y_2), \\
 &\phi^\pm_{ \eta^k}(y_1,0)=\int_{0}^{y_1}\eta^k(s)\de s, \\
 & \phi^\pm_{ \eta^k}(y_1,\pm m^{\pm})=g_{2}^{\pm,k}(y_1),\\
  &\p_{y_1}\phi^\pm_{ \eta^k}(L,y_2)=\omega_{ex}^k(\varphi_b^\pm(y_2)+\phi^\pm_{ \eta^k}(L,y_2)),\\
\end{aligned}
  \end{cases}
\end{equation}
where
   \begin{equation*}
      \begin{aligned}
      &f_i^{\pm,k}=W_i(\n\varphi_b^\pm, A_{b}^\pm,  B_{b}^\pm)-W_i(\n\varphi_b^\pm, \e
      A_{en}^{\pm,k}, \e B_{en}^{\pm,k}) , \quad i=1,2,\\
       &g_{1}^{\pm,k}(y_2)=\int_{0}^{y_2}\left(\frac{1}{\e J_{en}^{\pm,k}(s)}-\frac{1}{ J_b^\pm}\right)\de s,\\
    &g_{2}^{\pm,k}(y_1)=g^\pm(y_1)\pm 1+\int_{0}^{\pm m^{\pm}}\left(\frac{1}{\e J_{en}^{\pm,k}(s)}-\frac{1}{ J_b^\pm}\right)\de s.\\
    \end{aligned}
    \end{equation*}
Taking the difference of \eqref{5-21} and \eqref{5-7}, one can derive that
 \begin{equation}\label{5-22}
  \begin{cases}
  \begin{aligned}
  &\sum_{i,j=1,2}\p_{y_i}(a_{ij}(\n
   \phi_{\eta}^\pm,\e A_{en}^{\pm},\e B_{en}^{\pm})\p_{y_j}(\phi^\pm_{ \eta^k}-
   \phi_{\eta}^\pm)\\
  &=-\sum_{i,j=1,2}\p_{y_i}\left({a_{ij}(\n
   \phi_{\eta^k}^\pm,\e A_{en}^{\pm,k},\e B_{en}^{\pm,k})-a_{ij}(\n
   \phi_\eta^\pm,\e A_{en}^{\pm},\e B_{en}^{\pm})}\p_{y_j}\phi_{ \eta^k}^{\pm}\right)
   +\sum_{i=1}^2\p_{y_i}(f_i^{\pm,k}-f_i^{\pm}),\\
   &  (\phi^\pm_{ \eta^k}-
   \phi_{\eta}^\pm)(0,y_2)=g_{1}^{\pm,k}(y_2)-g_{1}^{\pm}(y_2),\\
  &(\phi^\pm_{ \eta^k}-
   \phi_{\eta}^\pm)(y_1,0)= \int_{0}^{y_1}(\eta^k(s)-\eta(s))\de s, \\
  &(\phi^\pm_{ \eta^k}-
   \phi_{\eta}^\pm)(y_1,\pm m^\pm)=g_{2}^{\pm,k}(y_1)-g_{2}^{\pm}(y_1), \\
   &\p_{y_1}(\phi^\pm_{ \eta^k}-
   \phi_{\eta}^\pm)(L,y_2)= \omega_{ex}^k(\varphi_b^\pm(y_2)+\phi^\pm_{ \eta^k}(L,y_2))-\omega_{ex}(\varphi_b^\pm(y_2)+\phi^\pm_{ \eta}(L,y_2)).\\
    \end{aligned}
    \end{cases}
  \end{equation}
 Then the estimate \eqref{4-11} implies  that
\begin{equation*}
\sum_{I=\pm}\|\phi_{\eta^k}^I- \phi_{\eta}^I\|_{2,\alpha;\mn^I}^{(-1-\alpha)}\leq \mc \left(\|\bm\zeta_0^k-\bm\zeta_0\|_{V_0}+\|\eta^k- {\eta}\|_{0,\alpha ;[0,L]}\right).
  \end{equation*}
  Therefore,
  \begin{equation}\label{5-23}
  \begin{aligned}
   &\|\mq(\bm\zeta_0^k,\eta^k)-\mq(\bm\zeta_0,\eta)\|
   _{0,\alpha ;[0,L]}\\
  &=\left\|\left(W_2(\n \varphi_b^++\n \phi_{\eta^k}^+ ,\e A_{en}^{+,k},\e B_{en}^{+,k} )-  W_2(\n \varphi_b^-+\n \phi_{\eta^k}^-,\e A_{en}^{-,k},\e B_{en}^{-,k} )\right)\right.\\
  &\quad\quad-\left.\left(W_2(\n \varphi_b^++\n \phi_{\eta}^+ ,\e A_{en}^+,\e B_{en}^+ )-  W_2(\n \varphi_b^-+\n \phi_{\eta}^-,\e A_{en}^-,\e B_{en}^- )\right)\right\| _{0,\alpha ;[0,L]}\\
  &\leq \mc\sum_{I=\pm}\left(\|\phi_{\eta^k}^I- \phi_{\eta}^I\|_{2,\alpha;\mn^I}^{(-1-\alpha)}
  +\|\bm\zeta_0^k-\bm\zeta_0\|_{V_0}\right)\leq  \mc \left(\|\bm\zeta_0^k-\bm\zeta_0\|_{V_0}+\|\eta^k- {\eta}\|_{0,\alpha ;[0,L]}\right),
  \end{aligned}
  \end{equation}
  which yields \eqref{5-20}.
  \par Next, we  prove the continuity of the map $ D_\eta\mq(\bm\zeta_0,\eta) $, i.e. to show that as $ k\rightarrow \infty $,
  \begin{equation}\label{5-24}
  D_\eta\mq(\bm\zeta_0^k,\eta^k)\rightarrow   D_\eta\mq(\bm\zeta_0,\eta), \quad {\rm{in}}\quad  V.
  \end{equation}
   It follows from \eqref{5-12} that the solutions $ \phi^{\pm,k}_{0}$ corresponding to $ (\bm\zeta_0^k,\eta^k) $ satisfy
  \begin{equation}\label{5-25}
  \begin{cases}
  \begin{aligned}
   &\sum_{i,j=1,2}\p_{y_i}(a_{ij}(\n
   \phi_{\eta^k}^\pm,\e A_{en}^{\pm,k},\e B_{en}^{\pm,k})\p_{y_j}\phi^{\pm,k}_{0})\\
  &=-\sum_{i,j=1,2}\p_{y_i}\left(\sum_{l=1}^2\p_{\p_{y_l}
   \phi_{\eta^k}^\pm}a_{ij}(\n
   \phi_{\eta^k}^\pm,\e A_{en}^{\pm,k},\e B_{en}^{\pm,k})\p_{y_l}
   \phi_0^{\pm,k} \p_{y_j}\phi_{ \eta^k}^{\pm}\right), \\
   &  \phi^{\pm,k}_{0}(0,y_2)=0,\\
  &\phi^{\pm,k}_{0}(y_1,0)=\int_{0}^{y_1}\eta_1(s)\de s, \\
  &\phi^\pm_{ \eta^k}(y_1,\pm m^{\pm})=0, \\
   &\p_{y_1}\phi^{\pm,k}_{0}(L,y_2)= (\omega_{ex}^k)^\prime(\varphi_b^\pm(y_2)+\phi_{ \eta^k}^{\pm}(L,y_2))\phi^{\pm,k}_{0}(L,y_2).\\
    \end{aligned}
    \end{cases}
  \end{equation}
Taking the difference of \eqref{5-25} and \eqref{5-12}, one has
  \begin{equation*}
  \begin{aligned}
  \sum_{I=\pm}\|\phi^{I,k}_{0}- \phi^{I}_{0}\|_{2,\alpha;\mn^I}^{(-1-\alpha)}
  \leq \mc\sum_{I=\pm}\left(\|\phi_{\eta^k}^I- \phi_{\eta}^I\|_{2,\alpha;\mn^I}^{(-1-\alpha)}
  +\|\bm\zeta_0^k-\bm\zeta_0\|_{V_0}\right).
  \end{aligned}
  \end{equation*}
  Combining \eqref{5-23} yields that
  \begin{equation}\label{5-26}
  \begin{aligned}
  \sum_{I=\pm}\|\phi^{I,k}_{0}- \phi^{I}_{0}\|_{2,\alpha;\mn^I}^{(-1-\alpha)}
  \leq  \mc \left(\|\bm\zeta_0^k-\bm\zeta_0\|_{V_0}+\|\eta^k- {\eta}\|_{0,\alpha ;[0,L]}\right).
  \end{aligned}
  \end{equation}
  Thus
  \begin{equation}\label{5-27}
  \begin{aligned}
  &\| D_\eta\mq(\bm\zeta_0^k,\eta^k)(\eta_1)-  D_\eta\mq(\bm\zeta_0,\eta)(\eta_1)\|_{0,\alpha ;[0,L]}\\
   &=\sum_{i=1}^2\left\|\left(\p_{\p_{y_i}\phi_{\eta^k}^{+}} W_2(\n \varphi_b^{+}+\n \phi^{+}_{\eta^k},\e A_{en}^{+,k},\e B_{en}^{+,k})\p_{y_i}\phi^{+,k}_{0}\right.\right.\\
  &\qquad\qquad\left.\left.-\p_{\p_{y_i}\phi^-_{\eta^k}}W_2(\n \varphi_b^{-}+\n \phi^{-}_{\eta^k},\e A_{en}^{-,k},\e B_{en}^{-,k})\p_{y_i}\phi^{-,k}_{0}\right)\right.\\
  &\qquad\quad-\left.\left(\p_{\p_{y_i}\phi_{\eta}^{+}} W_2(\n \varphi_b^{+}+\n \phi^{+}_{\eta},\e A_{en}^+,\e B_{en}^+)\p_{y_i}\phi^{+}_{0}\right.\right.\\
  &\qquad\qquad\left.\left.-\p_{\p_{y_i}\phi^-_{\eta}}W_2(\n \varphi_b^{-}+\n \phi^{-}_{\eta},\e A_{en}^-,\e B_{en}^-)\p_{y_i}\phi^{-}_{0}\right)\right\|
  _{0,\alpha ;[0,L]}\\
  &\leq \mc\sum_{I=\pm}\left(\|\phi^{I,k}_{0}- \phi^{I}_{0}\|_{2,\alpha;\mn^I}^{(-1-\alpha)}+ \|\phi^{I}_{\eta^k}- \phi^{I}_{\eta}\|
  _{2,\alpha;\mn^I}^{(-1-\alpha)}\right)+\mc\|\bm\zeta_0^k-\bm\zeta_0\|_{V_0}\\
  &\leq  \mc \left(\|\bm\zeta_0^k-\bm\zeta_0\|_{V_0}+\|\eta^k- {\eta}\|_{0,\alpha ;[0,L]}\right),
  \end{aligned}
  \end{equation}
    which implies that \eqref{5-24} holds.  
    \par In particular, at the background state,
    \begin{equation}\label{5-a}
  \begin{aligned}
 D_\eta\mq(\bm\zeta_b,0)(\eta_1)
  =\frac{  (c_b^+)^2(u_b^+)^3(\rho_b^+)^2}
{(c_b^+)^2-(u_b^+)^2}\p_{y_2}\phi^{+}_{b}(y_1,0)-\frac{  (c_b^-)^2(u_b^-)^3(\rho_b^-)^2}
{(c_b^-)^2-(u_b^-)^2}\p_{y_2}\phi^{-}_{b}(y_1,0).\\
  \end{aligned}
  \end{equation}
  Here $  \phi^{\pm}_b $ are the solutions of the following problems:
\begin{equation}\label{5-b}
\begin{cases}
e_1^\pm\p_{y_1}^2
\phi^{\pm}_b+e_2^\pm\p_{y_2}^2
 \phi^{\pm}_b=0,\\
\phi^{\pm}_b(0,y_2)=0,\\
\phi^{\pm}_b(y_1,0)=\int_0^{y_1}\eta_1(s)\de s,\\
\phi^{\pm}_b(y_1,\pm m^\pm)=0,\\
   \p_{y_1}\phi^{\pm}_b(L,y_2)=0,\\
  \end{cases}
\end{equation}
where
 \begin{equation*}
 e_{1}^\pm=u_b^\pm,\quad
 e_{2}^\pm=\frac{  (c_b^\pm)^2(u_b^\pm)^3(\rho_b^\pm)^2}
{(c_b^\pm)^2-(u_b^\pm)^2}.
\end{equation*}
 \par Define
   \begin{equation}\label{5-c}
 V_b=\{w\in V: w(0)=0\}.
  \end{equation}
  Then    $ D_\eta\mq(\bm\zeta_b,0) $ is  a continuous mapping from $  V$ to $  V_b \subset V $.
\par {\bf Step 2. The isomorphism of $ D_\eta\mq(\bm \zeta_b,0)$. }
\par To prove the isomorphism of $ D_\eta\mq(\bm\zeta_b,0)$, we need to show that for any given function $  P_\ast\in V_b $, there exists a unique $  \eta_\ast\in V  $ such that $ D_\eta\mq(\bm \zeta_b,0)(\eta_\ast)=P_\ast $, i.e.,
\begin{equation}\label{5-28}
\begin{aligned}
&P_\ast(y_1)= \frac{  (c_b^+)^2(u_b^+)^3(\rho_b^+)^2}
{(c_b^+)^2-(u_b^+)^2}\p_{y_2}\phi^{+}_\ast(y_1,0)-
\frac{  (c_b^-)^2(u_b^-)^3(\rho_b^-)^2}
{(c_b^-)^2-(u_b^-)^2}\p_{y_2}\phi^{-}_\ast(y_1,0).
\end{aligned}
\end{equation}
 It follows from \eqref{5-b} that  the solutions $  \phi^{\pm}_\ast $ satisfy
\begin{equation}\label{5-29}
\begin{cases}
e_1^\pm\p_{y_1}^2
\phi^{\pm}_\ast+e_2^\pm\p_{y_2}^2
 \phi^{\pm}_\ast=0,\\
\phi^{\pm}_\ast(0,y_2)=0,\\
\phi^{\pm}_\ast(y_1,0)=\int_0^{y_1}\eta_\ast(s)\de s,\\
\phi^{\pm}_\ast(y_1,\pm m^\pm)=0,\\
   \p_{y_1}\phi^{\pm}_\ast(L,y_2)=0.\\
  \end{cases}
\end{equation}

Let
\begin{equation*}
\begin{cases}
\hat y_1=y_1, \quad \hat y_2=\frac{y_2}{m^+}, \quad (y_1,y_2)\in \mn^+,\\
\hat y_1=y_1, \quad \hat y_2=\frac{y_2}{m^-}, \quad (y_1,y_2)\in \mn^-.\\
\end{cases}
\end{equation*}
Under the coordinates transformation, the domain  $ \mn^+\cup \mn^- $ is transformed into the following
 domain:
 \begin{equation*}
 \hat \mn^+\cup \hat \mn^-=\{(\hat y_1,\hat y_2): 0<\hat y_1 <L, 0<\hat y_2<1\}\cup \{(\hat y_1,\hat y_2): 0<\hat y_1 <L, -1<\hat y_2<0\}.
 \end{equation*}
The boundaries $ \Sigma_0^\pm,\Sigma_w^\pm, \Sigma_L^\pm $ become
\begin{equation*}
\begin{aligned}
 &\hat\Sigma_0^+=\{(\hat y_1,\hat y_2): \hat y_1=0, 0<\hat y_2<1\},\quad
 \hat\Sigma_0^-=\{(\hat y_1,\hat y_2): \hat y_1=0, -1<\hat y_2<0\},\\
 &\hat\Sigma_w^+=\{(\hat y_1,\hat y_2): 0<\hat y_1<L, \hat y_2=1\},\quad
 \hat\Sigma_w^-=\{(\hat y_1,\hat y_2): 0<\hat y_1<L, \hat y_2=-1\},\\
 &\hat\Sigma_L^+=\{(\hat y_1,\hat y_2): \hat y_1=L, 0<\hat y_2<1\}, \quad
 \hat\Sigma_L^-=\{(\hat y_1,\hat y_2): \hat y_1=L, -1<\hat y_2<0\}.
 \end{aligned}
 \end{equation*}
  We stretch $  \phi^{\pm}_\ast $ in the $ \hat y_2 $ direction and also flip $  \phi^{-}_\ast $ into $ \hat \mn^+ $ in the following way:
\begin{equation*}
 \hat\phi^{+}_\ast(\hat y_1,\hat y_2)=\phi^{+}_\ast\left(\hat y_1,
\sqrt{{e_2^+}/{e_1^+}}m^+\hat y_2\right),\quad  \hat\phi^{-}_\ast(\hat y_1,\hat y_2)=\phi^{-}_\ast\left(\hat y_1,
-\sqrt{{e_2^-}/{e_1^-}}m^-\hat y_2\right).
\end{equation*}
Then \eqref{5-29} can be written as
\begin{equation}\label{5-30}
\begin{cases}
\p_{\hat y_1 }^2
 \hat\phi^{\pm}_\ast+\p_{\hat y_2 }^2
 \hat\phi^{\pm}_\ast=0,
&\quad {\rm in}\quad \hat \mn^+,\\
\hat\phi^{\pm}_\ast(0, \hat y_2)=0, &\quad {\rm on}\quad      \hat \Sigma_0^+,\\
\hat\phi^{\pm}_\ast( \hat y_1,1)=0, &\quad {\rm on}\quad      \hat \Sigma_w^+,\\
\hat\phi^{\pm}_\ast( \hat y_1,0)=\int_0^{\hat y_1}\eta_\ast(s)\de s, &\quad {\rm on}\quad      \Sigma,\\
   \p_{\hat y_1}\hat\phi^{\pm}_\ast(L,  \hat y_2)=0, &\quad {\rm on}\quad
    \hat\Sigma_L^+.\\
  \end{cases}
\end{equation}
Therefore $ \hat\phi^{\pm}_\ast $ satisfy the Laplace equation with the same boundary conditions. By the uniqueness, we conclude that
$ \hat\phi^{+}_\ast= \hat\phi^{-}_\ast  $ in $ \hat \mn^+ $.
\par  Furthermore, \eqref{5-28} in the new  coordinates can be rewritten as
\begin{equation}\label{5-31}
\begin{aligned}
& P_\ast( \hat y_1)= \left(\sqrt{e_1^+e_2^+}+ \sqrt{e_1^-e_2^-}\right)\p_{\hat y_2}\hat\phi^{+}_\ast( \hat y_1,0).
\end{aligned}
\end{equation}
Thus we consider the following equation:
\begin{equation}\label{5-32}
\begin{cases}
\p_{\hat y_1 }^2
 \hat\phi^{+}_\ast+\p_{\hat y_2}^2
\hat\phi^{+}_\ast=0,
&\quad {\rm in}\quad \hat \mn^+,\\
\hat\phi^{+}_\ast(0, \hat y_2)=0, &\quad {\rm on}\quad      \hat \Sigma_0^+,\\
\p_{y_2}\hat\phi^{+}_\ast( \hat y_1,0)=\frac{1}{\sqrt{e_1^+e_2^+}+ \sqrt{e_1^-e_2^-}} P_\ast( \hat y_1), &\quad {\rm on}\quad      \Sigma,\\
\hat\phi^{+}_\ast(\hat y_1,1)=0, &\quad {\rm on}\quad      \hat \Sigma_w^+,\\
   \p_{\hat y_1}\hat\phi^{+}_\ast(L, \hat y_2)=0, &\quad {\rm on}\quad
    \hat\Sigma_L^+.\\
  \end{cases}
\end{equation}
Then similar arguments as in Lemma 4.2 yield that \eqref{5-32} has a unique solution  $ \hat\phi^{+}_\ast\in C_{2,\alpha ,\hat\mn^+}^{(-1-\alpha )} $ satisfying
\begin{equation}\label{5-33}
\|\hat\phi^{+}_\ast\| _{2,\alpha ,\hat\mn^+}^{(-1-\alpha )}\leq \mc \| P_\ast\|_{0,\alpha;[0,L]},
\end{equation}
where   $ \mc>0 $  depends  on $ ({\bm U}_b^+,{\bm U}_b^-) $, $ L$ and $ \alpha $.
\par Set $ \eta_\ast( \hat y_1)=\p_{\hat y_1}\hat\phi^{+}_\ast( \hat y_1,0) $,
then \eqref{5-33} shows that $ \eta_\ast\in V  $. Hence we have shown there exists a unique $ \eta_\ast\in V  $ such that $ D_\eta\mq(\bm\zeta_b,0)(\eta_\ast)= P_\ast $. The proof of the isomorphism of $ D_\eta\mq(\bm\zeta_b,0)$ is completed.
\section{Proof of Theorem 2.2}\noindent
\par Now, by the implicit function theorem, there exist  positive constants $ \sigma_5  $ and $ \mc $ depending only on $ (\bm {U}_b^+,\bm {U}_b^-) $, $L $ and $\alpha$ such that for $ \delta_4\leq \sigma_5 $, the equation $ \mq(\bm \zeta_0,\eta)=0 $ has a unique solution $ \eta=\eta(\bm\zeta_0) $ satisfying
\begin{equation}\label{6-1}
\|\eta\|_{0,\alpha;[0,L]}\leq \mc\|\bm \zeta_0-\bm \zeta_b\|_{V_0}= \mc\sigma_v.
\end{equation}
Here $ \delta_4 $ is defined in \eqref{5-3}.  Hence  the contact discontinuity curve $ g_{cd}(y_1)=\int_{0}^{y_1}\eta(s) \de s $ is determined  and  \eqref{6-1} implies that
\begin{equation}\label{6-2}
\|g_{cd}\|_{1,\alpha;[0,L]}\leq \mc_5\|\bm \zeta_0-\bm \zeta_b\|_{V_0}= \mc_5\sigma_v,
\end{equation}
where $ \mc_5>0 $ depends only on $ (\bm {U}_b^+,\bm {U}_b^-) $, $L $ and $\alpha$.
\par We choose $  \sigma_{cd}^\ast $ and  $\mc^\ast $ as
   \begin{equation}\label{6-3}
   \sigma_{cd}^\ast=\min\{\sigma_4,\sigma_5\} \quad {\rm{ and}} \quad \mc^\ast=4\mc_1+\mc_5,
   \end{equation}
   where $ \sigma_4 $  is defined in \eqref{5-17} and $ \mc_1 $ is  defined in \eqref{4-45}.
  Then if $ \sigma_v\leq\sigma_{cd}^\ast $,  $ \mathbf{Problem \ 3.1}  $ has  a unique piecewise smooth subsonic solution $ (\varphi^+,\varphi^-;g_{cd})$ satisfying
 \begin{equation*}
 \|\varphi^+ -\varphi_b^+\|_{2,\alpha;\mn^+}^{(-1-\alpha)}+\|\varphi^- -\varphi_b^-\|_{2,\alpha;\mn^-}^{(-1-\alpha)}+\| g_{cd} \|_{1,\alpha;[0,L]}\leq \mc^\ast\sigma_v.
 \end{equation*}
 Therefore
 \begin{equation}\label{6-4}
 \|\bm U^+ -\bm U_b^+\|_{1,\alpha;\mn^+}^{(-\alpha)}+\|\bm U^- -\bm U_b^-\|_{1,\alpha;\mn^-}^{(-\alpha)}+\| g_{cd} \|_{1,\alpha;[0,L]}\leq \mc^\ast\sigma_v.
 \end{equation}
 \par Since the coordinates transformation \eqref{3-4} is invertible, thus the solution transformed back in x-coordinates solves $ \mathbf{Problem \ 2.1}  $ and the estimate \eqref{6-4}
 implies that the estimates \eqref{2-13} and \eqref{2-14} hold in   Theorem 2.2. The proof of
  Theorem 2.2 is completed.
  \par {\bf Acknowledgement.} Weng is partially supported by National Natural Science Foundation of China  11971307, 12071359, 12221001.
  
\end{document}